\documentclass[12pt,reqno]{amsart}

\usepackage{a4wide}
\usepackage{amsmath}
\usepackage{psfrag}
\usepackage[T1]{fontenc}
\usepackage{amsthm}
\usepackage{amssymb,amsxtra}
\usepackage[usenames]{color}
\usepackage{stmaryrd}
\usepackage{mathrsfs}
\usepackage{upgreek}
\usepackage{ytableau}
\usepackage{comment,mathbbol}
\usepackage{paralist}
\usepackage{graphicx}
\usepackage{tikz}
\usetikzlibrary{matrix,fit}
\usepackage{geometry}
\DeclareSymbolFontAlphabet{\mathbbl}{bbold}

\newcommand{\NN}{\mathbb{N}}
\newcommand{\Z}{\mathbb{Z}}

\DeclareMathOperator{\res}{Res}
\DeclareMathOperator{\ind}{Ind}
\DeclareMathOperator{\Soc}{soc}
\DeclareMathOperator{\Rad}{rad}

\DeclareMathOperator{\Tr}{Tr}
\DeclareMathOperator{\Inf}{Inf}
\DeclareMathOperator{\Def}{Def}

\renewcommand{\leq}{\leqslant}
\renewcommand{\geq}{\geqslant}

\renewcommand{\unrhd}{\trianglerighteqslant}
\renewcommand{\b}{\boldsymbol}

\newcommand{\sym}[1]{\mathfrak{S}_{#1}}
\newcommand{\alt}[1]{\mathfrak{A}_{#1}}
\renewcommand{\Tr}{\mathrm{Tr}}
\newcommand{\ch}{\mathrm{ch}}
\renewcommand{\P}{\mathscr{P}}

\newcommand{\RP}{\mathscr{RP}}
\newcommand{\m}{{\pmb{\mathscr{M}}}}

\newcommand{\sgn}{\operatorname{\mathrm{sgn}}}
\newcommand{\C}{\mathscr{C}}
\newcommand{\Seq}{\mathrm{Seq}}

\newcommand{\N}{\mathrm{N}}

\newcommand{\Rho}{\boldsymbol{\rho}}
\renewcommand{\i}{\mathring{\imath}}
\renewcommand{\j}{\mathring{\jmath}}
\newcommand{\bb}[1]{\mathbbl{{#1}}}
\renewcommand{\t}{\mathfrak{t}}
\newcommand{\s}{\mathfrak{s}}
\newcommand{\List}{\mathrm{List}}
\newcommand{\ext}{\tiny{\text{$\bigwedge$}}}
\newcommand{\GL}{\mathrm{GL}}
\newcommand{\modcat}[1]{\text{${#1}$-mod}}
\newcommand{\sgnK}{\mathrm{K}^\pm}
\newcommand{\K}{\mathrm{K}}
\newcommand{\cont}{\text{{\tiny $\#$}}}
\newcommand{\suppL}[1]{\Lambda^\circ(#1)}

\newcommand{\G}{\mathcal{G}}
\newcommand{\Frob}{\mathscr{F}}
\newcommand{\trimt}[2]{#1^{\mathsf{T}_#2}}
\newcommand{\trimb}[2]{#1^{\mathsf{B}_#2}}
\newcommand{\trant}[2]{#1_{\leq_#2}}

\theoremstyle{definition}

\newtheorem{defn}{Definition}[section]
\newtheorem{rem}[defn]{Remark}

\newtheorem{eg}[defn]{Example}

\theoremstyle{plain}

\newtheorem{prop}[defn]{Proposition}
\newtheorem{lem}[defn]{Lemma}
\newtheorem{cor}[defn]{Corollary}
\newtheorem{thm}[defn]{Theorem}
\newtheorem*{Main Result A}{Main Result A}
\newtheorem*{Main Result B}{Main Result B}
\newtheorem*{Main Result C}{Main Result C}

\numberwithin{equation}{section}

\allowdisplaybreaks

\begin{document}
\title{On signed $p$-Kostka matrices}
\author{Eugenio Giannelli}
\address[E. Giannelli]{Dipartimento di Matematica e Informatica U. Dini, Viale Morgagni 67/a, Firenze, Italy.}
\email{address:eugenio.giannelli@unifi.it}

\author{Kay Jin Lim}
\address[K. J. Lim]{Division of Mathematical Sciences, Nanyang Technological University, SPMS-04-01, 21 Nanyang Link, Singapore 637371.}
\email{limkj@ntu.edu.sg}

\begin{abstract} We show that the signed $p$-Kostka numbers depend just on $p$-Kostka numbers and the multiplicities of projective indecomposable modules in certain signed Young permutation modules. We then examine the signed $p$-Kostka number $k_{(\alpha|\beta),(\lambda|p\mu)}$ in the case when $|\beta|=p|\mu|$. This allows us to explicitly describe the multiplicities of direct summands of a signed Young permutation module lying in the principal block of $F\sym{mp}$ in terms of the $p$-Kostka numbers.
\end{abstract}

\thanks{We would like to thank Susanne Danz for a fruitful discussion in the initial stage of this project. We are also grateful to the anonymous referee for useful suggestion and comments. The second author is supported by Singapore Ministry of Education AcRF Tier 1 grant RG17/20}.

\maketitle

\section{Introduction}\label{sec intro}
In 2001, Donkin introduced signed Young modules as the indecomposable summands of signed Young permutation modules \cite{Do}. In the same article, the author also defines listing modules as the indecomposable direct summands of the
mixed tensor product of symmetric and exterior powers of the natural module for the Schur algebras over fields of positive characteristics $p$.
These objects generalise important classical objects such as Young modules and tilting modules respectively.
The study of the decomposition of Young permutation modules into indecomposable summands (Young modules) is central in the modular representation theory of symmetric groups.
The Krull-Schmidt multiplicities of Young modules in Young permutation modules are called $p$-Kostka numbers. The complete determination of these multiplicities is well-known to be equivalent to the complete determination of the decomposition matrix for Schur algebras.

\medskip

%

In this article we focus on signed $p$-Kostka numbers. These numbers can be defined as the Krull-Schmidt multiplicities of signed Young modules into signed Young permutation modules, but appear in several other contexts.
For instance, signed $p$-Kostka numbers can also be regarded as multiplicities of listing modules in mixed tensor product of symmetric and exterior powers of the natural module \cite[3.1(3)]{Do}, or as the dimensions of weight spaces of irreducible modules for Schur superalgebras \cite[2.3(7)]{Do}.
%

Our first main result is Theorem \ref{T: Kostka reduction} below. Avoiding for the moment the introduction of the necessary techinical notation we summarise it as follows.

\begin{Main Result A}\label{MainA}
We give a closed formula for signed $p$-Kostka numbers. This reduces the computation of signed $p$-Kostka numbers to the knowledge of $p$-Kostka numbers and of projective signed $p$-Kostka numbers.
\end{Main Result A}

The formula announced above and given in Theorem \ref{T: Kostka reduction} generalizes and extends to signed $p$-Kostka numbers the well known Klyachko's multiplicity formula \cite{Klyachko}. Moreover, using our Main Result A we are able to deduce an analogue Steinberg Tensor Product Theorem for the irreducible modules for Schur superalgebras \cite[II.3.17]{Jantzen}.
This is done in Corollary \ref{C: steinberg}. We are aware that Corollary \ref{C: steinberg} can be directly obtained from Steinberg Tensor Product Theorem by counting the weight spaces. Nevertheless, we find interesting to highlight that our proof uses only techniques from the representation theory of symmetric groups.

%
%

\medskip

In Section 4 we analize the \textit{signed $p$-Kostka matrix} and we obtain the following result.
\begin{Main Result B}
We show that the signed $p$-Kostka matrix is lower unitriangular and has diagonal blocks given by the Kronecker products of some suitable $p$-Kostka matrices.
\end{Main Result B}

We refer the reader to Corollary \ref{C: signedrowremoval} for full details on our second main result. We remark that Corollary \ref{C: signedrowremoval} is obtained as a consequence of Theorem \ref{T: signedRowRemoval}.
Similarly, using Theorem \ref{T: signedRowRemoval} we are also able to explicitly describe the multiplicities of direct summands of signed Young permutation modules lying in the principal block of $F\sym{n}$, when $n$ is a multiple of $p$. This is done in Corollary \ref{C: principal block}.
%

\medskip

Signed Young permutation $F\sym{n}$-modules are naturally labelled by pairs of compositions $(\lambda |\mu)$ such that $|\lambda|+|\mu|=n$.
We devote the last part of the article to address the following problem.

\begin{Main Result C}
Given two pairs of compositions of $n$, we determine exactly when the corresponding signed Young permutation modules are isomorphic.
\end{Main Result C}

Our third main result is stated and proved in all details in Theorem \ref{T: char of M}.
The proof uses the signed Young Rule, recently discovered by Tan and the second author \cite{LT}.

%
%

\medskip

Our paper is organised as follows. In Section 2 we set up the notation and give some basic background that we will repeatedly use throughout the paper.
In Section \ref{S: Klyachko}, we prove our Main Result A.
In the first part of Section \ref{S: equal case} we generalized the row-removal formula for $p$-Kostka numbers obtained in \cite{BowGia} to the case of signed $p$-Kostka numbers.
The second part of Section \ref{S: equal case} is devoted to prove Theorem \ref{T: signedRowRemoval}, which in turn implies our Main Result B.
In Section \ref{S: not equal case}, we study Theorem \ref{T: signedRowRemoval} in a more general setting by dropping the assumption $|\beta|=p|\mu|$.
In Section 6, we prove Main Result C by addressing the problem of deciding when two signed Young permutation modules are isomorphic.

\section{Generalities}\label{S: pre}

Throughout this paper $F$ denotes a field of characteristic $p>0$. Whenever $A$ is a finite dimensional algebra over $F$,
an $A$-module is a finite-dimensional left $A$-module. All groups considered in this paper are finite groups.

Suppose that $V_1$ and $V_2$ are $A$-modules such that $V_1$ is isomorphic to a direct summand of $V_2$. We
write $V_1\mid V_2$. If $V_1$ is indecomposable then we denote by $[V_2:V_1]$ the Krull--Schmidt multiplicity of $V_1$
in $V_2$, that is, in every indecomposable direct sum decomposition of $V_2$, there are precisely $[V_2:V_1]$ direct summands that
are isomorphic to $V_1$.

Let $H$ be a subgroup of a group $G$ and, $V$ and $W$ be $FG$- and $FH$-modules respectively. We denote by $\res_H^G(V)$ the restriction of $V$ to $H$, and we denote by $\ind_H^G(W)$ the induction of $W$ to $G$. Suppose further that $V$ is indecomposable. A Green vertex $Q$ of $V$ is a minimal (with respect to inclusion) subgroup $Q$ of $G$ such that $V\mid \ind^G_Q\res^G_Q(V)$. In this case, a Green source with respect to $Q$ is an indecomposable $FQ$-module $S$ such that $V\mid \ind^G_QS$. 

If $H$ is normal in $G$ and $U$ is an $F[G/H]$-module then we denote by $\Inf_{G/H}^G(U)$ the inflation of $U$ to $G$. Similarly, if $H$ is a subgroup of $G$ acting trivially on an $FG$-module $V$, we denote by $\Def_{G/H}^GV$ the $F[G/H]$-module which is the same vector space $V$ with the action $(gH)\cdot v=gv$, i.e., the deflation of $V$ to $G/H$. Clearly, in this case, if $K\subseteq H$, then $\Def^{G/K}_{(G/K)/(H/K)}\Def^G_{G/K}V\cong \Def^G_{G/H}V$ via the identification $(G/K)/(H/K)\cong G/H$. Let $V$ and $W$ be $FG$- and $FH$-modules for some groups $G$ and $H$ respectively. The exterior tensor product $V\otimes_FW$ is an $F[G\times H]$-module in the obvious way. We shall denote this module by $V\boxtimes W$.

We have the following easy lemma.

\begin{lem}[{see \cite[Lemma 2.1]{GLOW}}]\label{L:trivialnormal} Let $V,W$ be $FG$-modules such that $W$ is indecomposable and suppose that $H$ is a normal subgroup of $G$ acting trivially on both the modules $V,W$. Then $[V:W]=[\Def^G_{G/H}V:\Def^G_{G/H}W]$.
\end{lem}

\subsection{Symmetric groups and their modules}\label{noth Sn}

Let $\NN_0,\NN$ be the sets of nonnegative integers and positive integers respectively. From now on and thereafter, all the infinite sums and infinite products encountered in this paper are indeed finite sums and finite products respectively. For instance, $\sum^\infty_{i=1}\delta_i$ means there exists $r\in\NN$ such that $\delta_i=0$ for all $i>r$ and hence $\sum^\infty_{i=1}\delta_i=\sum^r_{i=1}\delta_i$.

Let $n\in \NN_0$. We denote $\Seq(n)$ for the set consisting of infinite sequences of nonnegative integers $\delta=(\delta_i)_{i=1}^\infty$ such that $n=\sum^\infty_{i=1}\delta_i$. In this case, we write $|\delta|=n$ and denote the unique sequence $\delta$ such that $|\delta|=0$ by $\varnothing$. Let $\Seq=\bigcup_{n\in\NN_0}\Seq(n)$. Let $\delta,\gamma\in\Seq$ and $q,r\in\NN_0$. We define
\begin{align*}
  \delta+\gamma&=\gamma+\delta=(\delta_i+\gamma_i)^\infty_{i=1},\\
  q\cdot \delta&=(q\delta_i)^\infty_{i=1},\\
  \delta_{\leq r}&=(\delta_i)^r_{i=1},\\
  \delta_{>r}&=(\delta_{i+r})^\infty_{i=1}.
\end{align*}

Let $n\in\NN_0$. A composition of $n$ is a sequence of nonnegative integers $\alpha=(\alpha_1,\ldots,\alpha_r)$ for some $r\in\NN$ such that $n=\sum^r_{i=1}\alpha_i$. In this case, we call $\ell(\alpha)=r$ the length of $\alpha$ and, by abuse of notation, $|\alpha|=n$. By convention, we assume that $\alpha_i=0$ if $i>r$. The set of compositions of $n$ is denoted as $\C(n)$. The Young diagram $[\alpha]$ of $\alpha$ is the set \[[\alpha]=\{(i,j)\in\NN\times\NN:1\leq j\leq \alpha_i\}.\] A composition is called a partition if the terms are non-increasing and nonzero. The set of all partitions of $n$ is denoted as $\P(n)$. By abuse of notation, we also denote the unique partition of $0$ as $\varnothing$. The set $\P(n)$ is equipped with the usual dominance order $\unrhd$.

There is an obvious inclusion of $\P(n)$ into $\Seq(n)$ by adding zeroes to the tail. Let $\lambda\in\P(n)$, $\delta\in\Seq(n)$ and $\alpha,\beta\in\C(n)$. We may write $\lambda=\delta$ if $\lambda_i=\delta_i$ for all $i\in\NN$. If the length of a composition is out of the question, we may also write $\lambda=\alpha$ if $\lambda_i=\alpha_i$ for all $i\in\NN$. The  notions $q\cdot\alpha$ and $\alpha+\beta$ are defined similarly as before. The partition obtained from $\delta$ (respectively, $\alpha$) by rearranging its parts and deleting zero entries is denoted as $\wp(\delta)$ (respectively, $\wp(\alpha)$). Suppose further that $r=\ell(\alpha)$ and $s=\ell(\beta)$. We define
\begin{align*}
\alpha\cont\beta&=(\alpha_1,\ldots,\alpha_r,\beta_1,\ldots,\beta_s),\\
\alpha\cont\delta&=(\alpha_1,\ldots,\alpha_r,\delta_1,\delta_2,\ldots),\\
\alpha\cup\beta&=\wp(\alpha\cont\beta).
\end{align*}

Let $r\in\NN_0$, $\lambda\in\P(n)$ and $k=\ell(\lambda)$. We define the partitions
\begin{align*}
  \trimt{\lambda}{r}&=(\lambda_1,\ldots,\lambda_r),\\
  \trimb{\lambda}{r}&=(\lambda_{r+1},\lambda_{r+2},\ldots,\lambda_k),
\end{align*} if $0\leq r\leq k-1$ and, $\trimt{\lambda}{r}=\lambda$ and $\trimb{\lambda}{r}=\varnothing$ if $r\geq k$. Therefore, $\lambda=\trimt{\lambda}{r}\cont\trimb{\lambda}{r}$. A pair of partitions $(\alpha,\lambda)$, not necessarily of the same size, admits a horizontal $r$-row cut if $|\trimt{\alpha}{r}|=|\trimt{\lambda}{r}|$ for some $r\in\NN_0$.

Let $p$ be a prime number. A partition $\lambda$ is $p$-restricted if $0\leq \lambda_{i+1}-\lambda_i\leq p-1$ for all $i\in\NN$. A partition $\lambda$ is $p$-regular if and only if its conjugate $\lambda'$ is $p$-retricted. The set of $p$-restricted partitions of $n$ is denoted by $\RP(n)$. The sum \[\lambda=\sum_{i=0}^\infty p^i\cdot\lambda(i)\] is the $p$-adic expansion of $\lambda$ if, for each $i\in\NN_0$, $\lambda(i)$ is a $p$-restricted partition. For technical computation, we assume that $\lambda(-1)=\varnothing$. 

Let $n\in\NN_0$. We denote the symmetric group of degree $n$ by $\sym{n}$ and its alternating subgroup by $\alt{n}$. Let $\alpha\in\C(n)$. The standard Young subgroup of $\sym{n}$ labelled by $\alpha$ is denoted by $\sym{\alpha}=\prod^\infty_{j=1}\sym{\alpha_j}$, here the $j$th factor $\sym{\alpha_j}$ acts on the set \[\left \{1+\sum^{j-1}_{i=1}\alpha_i,2+\sum^{j-1}_{i=1}\alpha_i,\ldots,\sum^{j}_{i=1}\alpha_i\right \}.\] By our convention, if $k=\ell(\alpha)$, then $\sym{\alpha}=\sym{\alpha_1}\times\cdots\times \sym{\alpha_k}$. Let $G$ be a finite group. We identify the wreath product $G\wr\sym{\alpha}$ with the direct product $\prod^\infty_{j=1}G\wr\sym{\alpha_j}$. Furthermore, if $G\leq \sym{m}$ for some $m\in\NN_0$, we view $G\wr\sym{\alpha}$ as a subgroup of $\sym{mn}$ in the obvious way.

We recall, for instance from \cite[4.1.22, 4.1.24]{JK}, the structure of the Sylow $p$-subgroups of $\sym{n}$. If $n=\sum_{i=0}^\infty a_ip^i$ is the $p$-adic expansion of $n$ where $0\leq a_i\leq p-1$ then let $P_n$ be a fixed Sylow $p$-subgroup of the Young subgroup $\prod_{i=0}^\infty (\sym{p^i})^{a_i}$ of $\sym{n}$. Thus $P_n=\prod_{i=0}^\infty (P_{p^i})^{a_i}$ where $P_{p^i}$ is a Sylow $p$-subgroup of $\sym{p^i}$. Moreover, for each $i\in\NN_0$, the Sylow $p$-subgroup $P_{p^i}$ of $\sym{p^i}$ maybe identified with the $i$-fold iterated wreath product $C_p\wr \cdots \wr C_p$. Furthermore, we write $N_n=\N_{\sym{n}}(P_n)$. Whenever $\alpha\in\C(n)$, we denote by $P_\alpha$ the fixed Sylow $p$-subgroup of $\sym{\alpha}$ given by $P_\alpha=\prod_{i=1}^\infty P_{\alpha_i}$.

We use the notation $\rho=((p^i)^{n_i})_{i=0}^\infty$ to denote an element in $\Seq$ where the first $n_0$ terms of $\rho$ are 1, the next $n_1$ terms are $p$, and so on, so that $\rho\in\Seq(n)$ if $\sum^\infty_{i=0}n_ip^i=n$. By convention, if $n=0$ then $\rho=\varnothing$. Therefore we get $P_\rho=\prod_{i=0}^\infty(P_{p^i})^{n_i}$. In particular, for $i\in\NN_0$, if $|\rho|=n$, the group $P_\rho$ has precisely
$n_i$ orbits of sizes $p^i$ acting on the set $\{1,\ldots,n\}$. We write $N_\rho=\N_{\sym{n}}(P_\rho)$ and further identify $N_\rho$ with the following subgroup of $\sym{n}$: \[N_\rho=\prod_{i=0}^\infty(N_{p^i}\wr \sym{n_i})\subseteq\prod_{i=0}^\infty\sym{n_ip^i}.\]


We now describe some standard modules for the symmetric group algebras. The trivial and sign representations of $F\sym{n}$ are denoted by $F(n)$ and $\sgn(n)$ respectively (or simply $F$ and $\sgn$ if it is clear in the context). For each $\lambda\in\P(n)$, we denote by $S^\lambda$ the Specht $F\sym{n}$-module labelled by $\lambda$ as defined by James in \cite[Section 4]{James}. As $\lambda$ varies over the set of $p$-regular partitions of $n$, the quotient modules $D^\lambda:=S^\lambda/\Rad(S^\lambda)$ vary over a set of representatives of the isomorphism classes of simple $F\sym{n}$-modules. Analogously,
as $\mu$ varies over the set of $p$-restricted partitions of $n$ the submodules $D_\mu:=\Soc(S^\mu)$ vary over a set of representatives of the isomorphism classes of simple $F\sym{n}$-modules as well. If $\mu\in \RP(n)$ then one has \[D_{\m(\mu)}=D_\mu\otimes\sgn\cong D^{\mu'}\] where $\mu'$ is the conjugate partition of $\mu$ and $\m$ is the Mullineux map on the $p$-restricted partitions. We refer the reader to \cite{BrundanKujawa} for the explicit map $\m$.

Let $H$ be a subgroup of $\sym{n}$. We denote the trivial and sign representations of $F\sym{n}$ restricted to the subgroup $H$ by $F(H)$ and $\sgn(H)$ respectively. In the case when $H=\sym{\alpha}$ for some $\alpha\in\C(n)$, for simplicity, we write $F(\alpha)$ and $\sgn(\alpha)$ for $F(\sym{\alpha})$ and $\sgn(\sym{\alpha})$ respectively. Let $G$ be a finite group and suppose further that $n\geq 0$. 
Let $V$ be an $FG$-module. The $n$-fold tensor product of $V$ is an $F(G\wr\sym{n})$-module with the action given by \[(g_1,\ldots,g_n;\sigma)\cdot (v_1\otimes\cdots\otimes v_n)=\sgn(\sigma)g_1v_{\sigma^{-1}(1)}\otimes\cdots\otimes g_nv_{\sigma^{-1}(n)}\] for any $(g_1,\ldots,g_n;\sigma)\in G\wr\sym{n}$ and $v_1,\ldots,v_n\in V$. We denote this module by $\widehat{V}^{\otimes n}$. By convention, $\widehat{V}^{\otimes 0}$ is the trivial module for the trivial group. Furthermore, we write $\widehat{V}^{\otimes \alpha}=\boxtimes^\infty_{i=1}\widehat{V}^{\otimes \alpha_i}$ for the $F(G\wr\sym{\alpha})$-module. 

For further background on the representation theory of the symmetric group we refer the reader to \cite{James,JK}.

\subsection{Signed Young permutation modules, listing modules and signed $p$-Kostka numbers}\label{SS: sym module}  Let $\alpha,\beta$ be compositions such that $|\alpha|+|\beta|=n$. We call $(\alpha|\beta)$ a pair of compositions of $n$. Let $\C^2(n)$ be the set of all pairs of compositions of $n$.
Moreover, let $\P^2(n)$ be the set consisting of those pairs $(\alpha|\beta)\in\C^2(n)$ such that both $\alpha$ and $\beta$ are partitions. The set $\P^2(n)$ is partially ordered by $\unrhd$ where $(\alpha|\beta)\unrhd (\lambda|\mu)$ if, for all $k\in \NN$, we have both
\begin{enumerate}[(a)]
\item $\sum^k_{i=1}\alpha_i\geq \sum^k_{i=1}\lambda_i$, and
\item $|\alpha|+\sum^k_{i=1}\beta_i\geq |\lambda|+\sum^k_{i=1}\mu_i$.
\end{enumerate} In the case $(\alpha|\beta)\unrhd (\lambda|\mu)$ but $(\alpha|\beta)\neq(\lambda|\mu)$, we write $(\alpha|\beta)\rhd (\lambda|\mu)$.

For a simple example, in $\P^2(2)$, we have \[((2)|\varnothing)\rhd ((1,1)|\varnothing)\rhd((1)|(1))\rhd (\varnothing|(2))\rhd (\varnothing|(1,1)).\]

Let $p$ be an odd prime. We denote by $\P^2_p(n)$ the subset of $\P^2(n)$ consisting of all pairs of partitions of $n$ of the form $(\lambda|p\mu)$. We view $\C(n)$ and $\P(n)$ as subsets of $\C^2(n)$ and $\P^2(n)$, respectively, by identifying $\alpha$ with $(\alpha|\varnothing)$. The identification restricts the dominance order of $\P^2(n)$ to the usual dominance order of $\P(n)$ we introduced earlier.

Let $S_F(m,n)$ be the Schur algebra over $F$ as defined in \cite{Green}. The category of finite dimensional $S_F(m,n)$-modules $\modcat{S_F(m,n)}$ is Morita equivalent to the category of polynomial $\GL_m(F)$-modules of degree $n$. Furthermore, in the case when $m\geq n$, we have the Schur functor \[f:\modcat{S_F(m,n)}\to \modcat{F\sym{n}}\] defined by $f(V)=eV$ for certain idempotent $e$ of $S_F(m,n)$. Furthermore, the functor $f$ is exact. For simplicity, we assume $m\geq n$ throughout for the Schur algebra $S_F(m,n)$.

Let $E$ be the natural $\GL_m(F)$-module, and let $S^\alpha E$ and $\bigwedge^\beta E$ be the symmetric and exterior powers of $E$ with respect to the compositions $\alpha,\beta$, respectively; namely,
\begin{align*}
S^\alpha E&=S^{\alpha_1}E\otimes \cdots\otimes S^{\alpha_{\ell(\alpha)}}E,\\
\ext^\beta E&=\ext^{\beta_1}E\otimes\cdots\otimes \ext^{\beta_{\ell(\beta)}}E.
\end{align*} Let $K(\alpha|\beta)=S^\alpha E\otimes \bigwedge^\beta E$ for each $(\alpha|\beta)\in \C^2(n)$. Then $K(\alpha|\beta)$ is a polynomial $\GL_m(F)$-module of degree $n$. By \cite[\S3]{Do}, the indecomposable summands of the mixed powers $K(\alpha|\beta)$ are called the listing modules and their isomorphism classes are labelled by the set $\P^2_p(n)$. The listing module labelled by $(\lambda|p\mu)\in\P^2_p(n)$ is denoted by $\List(\lambda|p\mu)$. Furthermore, \[K(\lambda|p\mu)\cong \List(\lambda|p\mu)\oplus C(\lambda|p\mu)\] where $C(\lambda|p\mu)$ is a direct sum of listing modules labelled by pairs $(\xi|p\zeta)\in\P^2_p(n)$ such that $(\xi|p\zeta)\rhd (\lambda|p\mu)$.

For $(\alpha|\beta)\in\C^2(n)$, one defines the signed Young permutation $F\sym{n}$-module \[M(\alpha|\beta):=\ind_{\sym{\alpha}\times\sym{\beta}}^{\sym{n}}(F(\alpha)\boxtimes \sgn(\beta))\,.\]
In the case when $\beta=\varnothing$ this yields the usual Young permutation $F\sym{n}$-module $M^\alpha=M(\alpha|\varnothing)$. Notice that $M(\alpha|\beta)\otimes\sgn\cong M(\beta|\alpha)$ and $M(\alpha|\beta)\cong M(\wp(\alpha)|\wp(\beta))$. (We refer the reader to Section 2.1 for the definition of $\wp(\alpha)$.)
 Therefore, without loss of generality, up to isomorphism, we may only consider $M(\alpha|\beta)$ when $(\alpha|\beta)\in\P^2(n)$. For $\gamma,\delta\in\Seq$, suppose that $\alpha=\gamma$ and $\beta=\delta$, we write $M(\gamma|\delta)$ for $M(\alpha|\beta)$ (and hence $M^\gamma$ for $M^\alpha$).

For each $(\lambda|p\mu)\in\P^2_p(n)$, let $Y(\lambda|p\mu)=f(\List(\lambda|p\mu))$ where $f$ is the Schur functor. By convention, $Y(\varnothing|\varnothing)$ is the trivial $F\sym{0}$-module. Following \cite[\S3]{Do}, we have $f(K(\alpha|\beta))\cong M(\alpha|\beta)$. Since $f$ is exact, we obtain that $\{Y(\lambda|p\mu):(\lambda|p\mu)\in\P^2_p(n)\}$ is a set consisting of the isomorphism classes of the indecomposable summands of the signed Young permutation modules. We call $Y(\lambda|p\mu)$ a signed Young permutation module. As such, when $(\lambda|p\mu)$ varies over the set $\P^2_p(n)$, both subsets consisting of the isomorphism classes containing $M(\lambda|p\mu)$ and $Y(\lambda|p\mu)$, respectively,  are bases for the Green ring for $F\sym{n}$ generated by all signed Young permutation modules.

In the case when $\mu=\varnothing$, one obtains $Y(\lambda|\varnothing)\cong Y^\lambda$, the usual indecomposable Young $F\sym{n}$-module labelled by  the partition $\lambda$ of $n$. The indecomposable Young $F\sym{n}$-module $Y^\lambda$ is projective if and only if $\lambda$ is a $p$-restricted partition. In this case, $Y^\lambda$ is a projective cover of the simple $F\sym{n}$-module $D_\lambda$.

For the remainder of this paper, we fix a total order on $\P^2_p(n)$ as follows. For each $n\in \NN_0$, we fix a total order $\succeq_n$ on the set $\P(n)$ refining the dominance order $\unrhd$ on $\P(n)$. For instance, we may use $\succeq_n$ as the dictionary order on $\P(n)$ (see \cite[3.4]{James}). Let $\succeq$ be a total order on the set $\P^2_p(n)$ refining the dominance order on $\P^2_p(n)$ such that $(\alpha|p\beta)\succeq (\delta|p\beta)$ if and only if $\alpha\succeq_m \delta$ where $|\alpha|=m=|\delta|$.

We remark that, since $\succeq$ refines $\unrhd$, if $(\alpha|p\beta)\succeq (\lambda|p\mu)\succeq (\delta|p\beta)$, then $(\alpha|p\beta)\unrhd (\lambda|p\mu)\unrhd (\delta|p\beta)$ and hence $\beta\unrhd \mu\unrhd \beta$ which implies that $\beta=\mu$. Also, the restriction of $\succeq$ on $\P^2_p(n)$ to the subset $\P(n)$ is precisely $\succeq_n$.

We now define the signed $p$-Kostka numbers and signed $p$-Kostka matrices.

\begin{defn}\label{defn Kostka}\
\begin{enumerate}[(i)]
\item Let $(\alpha|\beta),(\lambda|p\mu)\in\P^2(n)$. The Krull--Schmidt multiplicity
$$k_{(\alpha|\beta),(\lambda|p\mu)}=[M(\alpha|\beta):Y(\lambda|p\mu)]=[K(\alpha|\beta):\List(\lambda|p\mu)]$$
is called a signed $p$-Kostka number. If $\beta=\varnothing=\mu$ then \[k_{\alpha,\lambda}=[M^\alpha:Y^\lambda]=[S^\alpha E:P(\lambda)]\] is
called a $p$-Kostka number where $P(\lambda)$ is the projective cover of the simple $S_F(m,n)$-module $L(\lambda)$ of highest weight $\lambda$ and $m\geq n$. By convention, $k_{(\varnothing|\varnothing),(\varnothing|\varnothing)}=1=k_{\varnothing,\varnothing}$. Furthermore, if $\delta,\gamma\in\Seq$ such that $\wp(\delta)=\alpha$ and $\wp(\gamma)=\beta$, then we write \[k_{(\delta|\gamma),(\lambda|p\mu)}=k_{(\alpha|\beta),(\lambda|p\mu)}.\]

\item The signed $p$-Kostka matrix $\sgnK_n$ is the square matrix whose entries are signed $p$-Kostka numbers of the form $k_{(\alpha|p\beta),(\lambda|p\mu)}$ where $(\alpha|p\beta),(\lambda|p\mu)\in\P^2_p(n)$ with respect to the fixed total order $\succ$ on $\P^2_p(n)$. The $p$-Kostka matrix $\K_n$ is the square matrix whose entries are $p$-Kostka numbers $k_{\alpha,\lambda}$ where $\alpha,\lambda\in\P(n)$ with respect to the total order $\succ_n$ on $\P(n)$.
\end{enumerate}
\end{defn}

\begin{rem}\label{Rem: Schur super}\
  \begin{enumerate}[(i)]
    \item By \cite[\S2.3]{Do}, the number $k_{(\alpha|\beta),(\lambda|p\mu)}$ is also equal to the dimension of the $(\alpha|\beta)$-weight space $L(\lambda|p\mu)^{(\alpha|\beta)}$ of the irreducible module $L(\lambda|p\mu)$ of highest weight $(\lambda|p\mu)$ of the Schur superalgebra $S_F(a|b,n)$ where $a,b\geq n$.
    \item By definition, it is clear that $\K_n$ is the top left $(\ell\times \ell)$-submatrix of $\sgnK_n$ where $\ell=|\P(n)|$.
    \item In \cite{Lim}, for each $(\alpha|\beta)\in\P^2(n)$, the second author gave an explicit formula to write $M(\alpha|\beta)$ as a linear combination of $M(\zeta|p\eta)$'s for $(\zeta|p\eta)\in\P^2_p(n)$ in the Green ring of the group algebra $F\sym{n}$. Therefore, theoretically, all signed $p$-Kostka numbers are determined whenever the signed $p$-Kostka numbers of the form $k_{(\zeta|p\eta),(\lambda|p\mu)}$ are known. Therefore $\sgnK_n$ is only labelled by $\P^2_p(n)\times \P^2_p(n)$.
  \end{enumerate}
\end{rem}

\begin{eg}\label{Eg: signKostka6} Let $n=6$, $p=3$ and use the dictionary order. Using Magma \cite{Magma}, we compute and depict below the signed $p$-Kostka matrix $\sgnK_6$. The submatrices in red, blue and yellow colors are $\K_6$, $\K_3$ and $\K_2$ respectively. This structural property is explained in full generality in Section 4 by Corollary \ref{C: signedrowremoval}(iii) (by our convention, $\K_0=\begin{pmatrix}1\end{pmatrix}$).
\small{\begin{center}
  \begin{tikzpicture}[row sep=-5pt,column sep=-1pt]
    \matrix(m)[matrix of math nodes,ampersand replacement=\&,column 1/.style={font=\tiny, anchor=base east}]
    {((6)|\varnothing)\ \ \&1\& 0\& 0\& 0\& 0\& 0\& 0\& 0\& 0\& 0\& 0\& 0\& 0\& 0 \&0 \&0\\
    ((5,1)|\varnothing)\ \ \&0\& 1\& 0\& 0\& 0\& 0\& 0\& 0\& 0\& 0\& 0\& 0\& 0\& 0\& 0\& 0\\
    ((4,2)|\varnothing)\ \ \&0\& 1\& 1\& 0\& 0\& 0\& 0\& 0\& 0\& 0\& 0\& 0\& 0\& 0\& 0\& 0\\
    ((4,1^2)|\varnothing)\ \ \&0\& 1\& 1\& 1\& 0\& 0\& 0\& 0\& 0\& 0\& 0\& 0\& 0\& 0\& 0\& 0\\
    ((3^2)|\varnothing)\ \ \&1\& 0\& 1\& 0\& 1\& 0\& 0\& 0\& 0\& 0\& 0\& 0\& 0\& 0\& 0\& 0\\
    ((3,2,1)|\varnothing)\ \ \&0\& 1\& 2\& 0\& 0\& 1\& 0\& 0\& 0\& 0\& 0\& 0\& 0\& 0\& 0\& 0\\
    ((3,1^3)|\varnothing)\ \ \&0\& 1\& 3\& 1\& 0\& 1\& 1\& 0\& 0\& 0\& 0\& 0\& 0\& 0\& 0\& 0\\
    ((2^3)|\varnothing)\ \ \&0\& 0\& 3\& 0\& 0\& 1\& 0\& 1\& 0\& 0\& 0\& 0\& 0\& 0\& 0\& 0\\
    ((2^2,1^2)|\varnothing)\ \ \&0\& 0\& 4\& 0\& 0\& 2\& 1\& 1\& 1\& 0\& 0\& 0\& 0\& 0\& 0\& 0\\
    ((2,1^4)|\varnothing)\ \ \&0\& 0\& 6\& 0\& 0\& 3\& 3\& 1\& 3\& 1\& 0\& 0\& 0\& 0\& 0\& 0\\
    ((1^6)|\varnothing)\ \ \&0\& 0\& 9\& 0\& 0\& 4\& 6\& 1\& 9\& 4\& 1\& 0\& 0\& 0\& 0\& 0\\
    ((3)|(3))\ \ \&0\& 0\& 0\& 0\& 0\& 0\& 0\& 0\& 0\& 0\& 0\& 1\& 0\& 0\& 0\& 0\\
    ((2,1)|(3))\ \ \&0\& 0\& 0\& 0\& 0\& 0\& 1\& 0\& 1\& 0\& 0\& 0\& 1\& 0\& 0\& 0\\
    ((1^3)|(3))\ \ \&0\& 0\& 0\& 0\& 0\& 0\& 1\& 0\& 3\& 1\& 0\& 0\& 1\& 1\& 0\& 0\\
    (\varnothing|(6))\ \ \&0\& 0\& 0\& 0\& 0\& 0\& 0\& 0\& 0\& 0\& 0\& 0\& 0\& 0\& 1\& 0\\
    (\varnothing|(3^2))\ \ \&0\& 0\& 0\& 0\& 0\& 0\& 0\& 0\& 1\& 0\& 0\& 0\& 0\& 0\& 1\& 1\\
    };
  \node[fit=(m-1-2)(m-11-12),draw=red,thick,inner sep=-1pt]{};
  \node[fit=(m-12-13)(m-14-15),draw=blue,thick,inner sep=-1pt]{};
  \node[fit=(m-15-16)(m-16-17),draw=yellow,thick,inner sep=-1pt]{};
  \end{tikzpicture}
\end{center}}
\end{eg}

We end this subsection with the following general row removal formula was proved by Bowman and the first author. We remark that the case when $r=1$ was first proved by Fang-Henke-Knoenig (see \cite[Corollary 9.1]{FangHenkeKoenig}).

\begin{thm}[{\cite[Corollary 1.1]{BowGia}}]\label{T: RowRemoval} Let $\alpha,\lambda\in\P(n)$ such that the pair $(\alpha,\lambda)$ admits a horizontal $r$-row cut. Then \[k_{\alpha,\lambda}=k_{\trimt{\alpha}{r},\trimt{\lambda}{r}}k_{\trimb{\alpha}{r},\trimb{\lambda}{r}}.\]
\end{thm}




\subsection{Brauer correspondence for $p$-permutation modules}

We recall now the definition and the basic properties of the Brauer quotients for $FG$-modules. Let $G$ be a finite group. Given an $FG$-module $V$ and $Q$ a $p$-subgroup of $G$, the set of fixed points $V$ of $Q$ is denoted by \[V^Q=\{v\in V\ :\ gv=v\ \text{for all}\ g\in Q \}.\] It is easy to see that $V^Q$ is an
$F\N_G(Q)$-module on which $Q$ acts trivially. For $P$ a proper subgroup of $Q$, the relative trace map $\Tr_P^Q:V^P\rightarrow V^Q$ is the linear map defined by $$\Tr_P^Q(v)=\sum_{g\in Q/P}gv,$$ where $Q/P$ is a left transversal of $P$ in $Q$. The map is independent of the choice of the left transversal. We observe that $$\Tr^Q(V)=\sum_{P<Q}\Tr_P^Q(V^P)$$ is an $F\N_G(Q)$-module on which $Q$ acts trivially. Therefore we can define the $F[\N_G(Q)/Q]$-module called the Brauer quotient of $V$ with respect to $Q$ by \[V(Q)=V^Q/\Tr^Q(V).\]

An $FG$-module $V$ is called a $p$-permutation module if for every Sylow $p$-subgroup $P$ of $G$ there exists an $F$-linear basis $\mathcal{B}_P$ of $V$ that is permuted by $P$, or equivalently, $V$ has trivial Green source. It follows from the definition that the class of $p$-permutation modules is closed under taking direct sum, direct summand, tensor product, restriction and induction. If $V$ is an indecomposable $FG$-module and $Q$ is a $p$-subgroup of $G$, then $V(Q)\neq 0$ implies that $Q$ is contained in a Green vertex of $V$. Brou\'e proved in \cite{Broue} that the converse holds in the case of $p$-permutation modules.

\begin{thm}[{\cite[Theorem 3.2]{Broue}}]\label{BT2}
Let $V$ be an indecomposable $p$-permutation module and $Q$ be a Green vertex of $V$. Let $P$ be a $p$-subgroup of $G$, then $V(P)\neq 0$ if and only if $P\leq {}^gQ$ for some $g\in G$.
\end{thm}

Another important tool we will be using to study the signed $p$-Kostka numbers is the Brauer correspondence of $p$-permutation modules developed by Brou\'e as follows.

\begin{thm}[{\cite[Theorem 3.2 and 3.4]{Broue}}]\label{BC1}
An indecomposable $p$-permutation module $V$ has Green vertex $Q$ if and only if $V(Q)$ is a projective $F[\N_G(Q)/Q]$-module. Furthermore, we have the following statements.
\begin{enumerate}[(i)]
\item The Brauer map sending $V$ to $V(Q)$ is a
bijection between the isomorphism classes of the indecomposable $p$-permutation
$FG$-modules with Green vertex $Q$ and the isomorphism classes of the indecomposable
projective $F[\N_G(Q)/Q]$-modules. Furthermore, $\Inf^{\N_G(Q)}_{\N_G(Q)/Q}V(Q)$ is the
Green correspondent of $V$ with respect to the subgroup $\N_G(Q)$.
\item Let $V$ be a $p$-permutation $FG$-module and let $U$ be an indecomposable $FG$-module with Green vertex $Q$. Then $U$ is a direct summand of $V$ if and only if $U(Q)$ is a direct summand of $V(Q)$. Moreover, \[[V:U]=[V(Q):U(Q)].\]
\end{enumerate}
\end{thm}


\subsection{Brauer correspondences for signed Young and signed Young permutation modules}

In this subsection, we describe the Brauer correspondents of signed Young and signed Young permutation modules. Recall that, for each $k\in \NN_0$, $P_k$ is the fixed Sylow $p$-subgroup of $\sym{k}$ and $N_k=\N_{\sym{k}}(P_k)$.

\begin{defn}\label{D: bifunctor} Let $p$ be an odd prime, $s,n\in \NN_0$ such that $s\leq n$, $U$ and $V$ be $F\sym{s}$- and $F\sym{n-s}$-modules respectively, and let $X$ be an $FG$-module. We define the $F[G\wr\sym{n}]$-module \[\G_X(U,V)=\ind^{G\wr\sym{n}}_{G\wr(\sym{s}\times\sym{n-s})} \left ((\Inf^{G\wr\sym{s}}_{\sym{s}}U)\boxtimes ((\Inf^{G\wr\sym{n-s}}_{\sym{n-s}}V)\otimes \widehat{X}^{\otimes n-s})\right ).\] For each $k\in\NN_0$, we write \[\G_k(U,V)=\G_{\Def^{N_k}_{N_k/P_k}\sgn(N_k)}(U,V).\]
\end{defn}

Notice that, if $H$ acts trivially on $X$, then subgroup $H^{n}$ of the base group of $G\wr\sym{n}$ acts trivially on $\G_X(U,V)$ and we have $\G_{\Def^G_{G/H}X}(U,V)\cong \Def^{G\wr\sym{n}}_{(G/H)\wr\sym{n}}\G_X(U,V)$.  Let $\alt{k}$ be the alternating subgroup of $\sym{k}$. For $k\geq 2$, since $\N_{\alt{k}}(P_k)$ acts trivially on $\sgn(N_k)$ and via the identification $(N_k/P_k)/(\N_{\alt{k}}(P_k)/P_k)\cong C_2$, we therefore have
\begin{align}\label{Eq: 1}
\Def^{(N_k/P_k)\wr\sym{n}}_{((N_k/P_k)/(\N_{\alt{k}}(P_k)/P_k))\wr\sym{n}}\G_k(U,V)&\cong  \G_{\Def^{N_k}_{(N_k/P_k)/(\N_{\alt{k}}(P_k)/P_k)}\sgn(N_k)}(U,V)\notag\\
&\cong \G_{\sgn(C_2)}(U,V)\cong \G_2(U,V).
\end{align}



\begin{defn}\label{D: module W} Let $(\alpha|\beta)\in \C^2(n)$ and $\rho=((p^i)^{n_i})^\infty_{i=0}\in\Seq(n)$.
\begin{enumerate}[(i)]
  \item We denote by $\Lambda((\alpha|\beta),\rho)$ the set of tuples $(\boldsymbol{\gamma}|\boldsymbol{\delta})=((\b{\gamma}^{(i)})^\infty_{i=0}|(\b{\delta}^{(i)})^\infty_{i=0})$ such that
\begin{enumerate}[(a)]
\item for each $i\in\NN_0$, $\b{\gamma}^{(i)},\b{\delta}^{(i)}\in\Seq$,
\item $\alpha=\sum_{i=0}^\infty p^i\cdot\b{\gamma}^{(i)}$, $\beta=\sum_{i=0}^\infty p^i\cdot \b{\delta}^{(i)}$, and
\item for each $i\in\NN_0$, $|\b{\gamma}^{(i)}|+|\b{\delta}^{(i)}|=n_i$.
\end{enumerate} By our convention, $\Lambda((\varnothing|\varnothing),\varnothing)=\{(\pmb{\varnothing}|\pmb{\varnothing})\}$ where $\pmb{\varnothing}=(\varnothing)^\infty_{i=1}$.
  \item For each $(\boldsymbol{\gamma}|\boldsymbol{\delta})\in \Lambda((\alpha|\beta),\rho)$, we define the $F[N_\rho/P_\rho]$-module \[W(\b{\gamma}|\b{\delta})=\boxtimes^\infty_{i=0}\G_{p^i}(M^{\b{\gamma}^{(i)}},M^{\b{\delta}^{(i)}}).\] Again, we observe that this is indeed a finite outer tensor product because, for $p^i>n$, we have $\b{\gamma}^{(i)}=\varnothing=\b{\delta}^{(i)}$.
\end{enumerate}
\end{defn}


Given $(\alpha|\beta)$ and $\rho$ as in the definition above, the following describes the structure of the Brauer quotient $M(\alpha|\beta)(P_\rho)$.

\begin{prop}[{\cite[Proposition 3.12]{GLOW}}]\label{P: Brauer M}
Let $\rho=((p^i)^{n_i})^\infty_{i=0}\in\Seq(n)$ and let $(\alpha|\beta)\in \C^2(n)$. Then we have an isomorphism of $F[N_\rho/P_\rho]$-modules $$M(\alpha|\beta)(P_\rho)\cong \bigoplus_{(\b{\gamma}|\b{\delta})\in \Lambda((\alpha|\beta),\rho)}W(\b{\gamma}|\b{\delta})\,.$$
\end{prop}

To describe the Brauer correspondents of signed Young modules, we introduce the following notation.
We invite the reader to recall the definition of $\P^2_p(n)$, given at the start of Section \ref{SS: sym module}.

\begin{defn}\label{D: rho} For each $(\lambda|p\mu)\in \P^2_p(n)$, let
\begin{equation*}\label{eqn padic}
\lambda=\sum_{i=0}^\infty p^i\cdot \lambda(i)\quad \text{ and }\quad \mu=\sum_{i=0}^\infty p^i\cdot \mu(i)\,,
\end{equation*} be the $p$-adic expansions of $\lambda$ and $\mu$ respectively and let $n_i=|\lambda(i)|+|\mu(i-1)|$ for all $i\in\NN_0$ (recall that, by convention, $\mu(-1)=\varnothing$).
\begin{enumerate}[(i)]
\item Let $\Rho:\P^2_p(n)\to \Seq(n)$ be the function defined by $\Rho(\lambda|p\mu)=((p^i)^{n_i})^\infty_{i=0}\in\Seq(n)$.
\item Let $\rho=\Rho(\lambda|p\mu)$. We define the $F[N_\rho/P_\rho]$-module
\begin{align*}
\mathbf{Q}(\lambda|p\mu):&=\boxtimes_{i=0}^\infty \G_{p^i}(Y^{\lambda(i)},Y^{\mu(i-1)}).
\end{align*}
\end{enumerate}
\end{defn}

By \cite[5.1(3) and 5.2(2)]{Do} and \cite[Theorem~3.8(b) and Corollary 3.18]{DL}, we obtain the following result.

\begin{thm}\label{T: Young Broue cor} Let $(\lambda|p\mu)\in\P^2(n)$ and $\rho=\Rho(\lambda|p\mu)$. Then
\begin{enumerate}[(i)]
\item $Y(\lambda|p\mu)$ has Green vertex $P_\rho$, and
\item $Y(\lambda|p\mu)(P_\rho)\cong \mathbf{Q}(\lambda|p\mu)$ as $F[N_\rho/P_\rho]$-modules.
\end{enumerate}
\end{thm}

\section{Klyachko's multiplicity formula for signed $p$-Kostka numbers}\label{S: Klyachko}

Let $p$ be an odd prime. We are now ready to exploit the Brauer correspondences for signed Young permutation and signed Young modules obtained in Proposition \ref{P: Brauer M} and Theorem \ref{T: Young Broue cor} to deduce the main result of this section (see Theorem \ref{T: Kostka reduction}) which is the signed version of the Klyachko's multiplicity formula. We then use the result to deduce Steinberg Tensor Product Theorem for the irreducible modules for the Schur superalgebras.

We begin with a lemma.


\begin{lem}\label{L: module V,S} Let $\gamma,\delta$ be partitions and $\lambda,\mu$ be $p$-restricted partitions such that $n=|\lambda|+|\mu|=|\gamma|+|\delta|$. Then \[[\G_2(M^\gamma,M^\delta):\G_2(Y^\lambda,Y^\mu)]=\left \{\begin{array}{ll}k_{\gamma,\lambda}k_{\delta,\mu}&\text{if $|\gamma|=|\lambda|$ and $|\delta|=|\mu|$,}\\ 0 &\text{otherwise.}\end{array}\right .\]
\end{lem}
\begin{proof} Let $M^\gamma\cong \bigoplus_{\alpha\unrhd\gamma} k_{\gamma,\alpha}Y^\alpha$ and $M^\delta\cong \bigoplus_{\beta\unrhd\delta} k_{\delta,\beta}Y^\beta$. Then, by Definition \ref{D: bifunctor}, \[\G_2(M^\gamma,M^\delta)\cong \bigoplus_{\substack{\alpha\unrhd\gamma,\\ \beta\unrhd\delta}} k_{\gamma,\alpha}k_{\delta,\beta}\cdot \G_2(Y^\alpha,Y^\beta).\] Since $\lambda,\mu$ are $p$-restricted partitions, by \cite[Lemma 4.4 and Proposition 4.5]{GLOW}, we see that $\G_2(Y^\lambda,Y^\mu)\cong \G_2(Y^\alpha,Y^\beta)$ if and only if $\lambda=\alpha$ and $\mu=\beta$. Therefore, we obtain our desired equality. Alternatively, using \cite[Propositions 1.2 and 5.1]{BKuelshammer}, we see that each $F[C_2\wr\sym{n}]$-module $\G_2(Y^\alpha,Y^\beta)$ is indecomposable. The complexity of $\G_2(Y^\alpha,Y^\beta)$ is the sum of the complexities of $Y^\alpha$ and $Y^\beta$ and therefore is a projective module (i.e., has complexity zero) if and only if both $\alpha$ and $\beta$ are $p$-restricted. Now use \cite[A.1]{DL}.
\end{proof}

In view of Lemma \ref{L: module V,S}, we consider the following subset of $\Lambda((\alpha|\beta),\Rho(\lambda|p\mu))$ (recall Definitions \ref{D: module W} and \ref{D: rho}). 

\begin{defn}\label{D: suppL} Suppose that $(\alpha|\beta)\in\C^2(n)$ and $(\lambda|p\mu)\in\P^2_p(n)$. Let $\suppL{(\alpha|\beta),(\lambda|p\mu)}$ be the subset of $\Lambda((\alpha|\beta),\Rho(\lambda|p\mu))$ consisting of $(\b{\gamma}|\b{\delta})$ such that, for all positive integers $i$,
\begin{enumerate}[(a)]
\item $|\b{\gamma}^{(i)}|=|\lambda(i)|$ and $\lambda(i)\unrhd \wp(\b{\gamma}^{(i)})$,
\item $|\b{\delta}^{(i)}|=|\mu(i-1)|$ and $\mu(i-1)\unrhd \wp(\b{\delta}^{(i)})$, and
\item $(\lambda(0)|\varnothing)\unrhd (\wp(\b{\gamma}^{(0)})|\wp(\b{\delta}^{(0)}))$.
\end{enumerate}
\end{defn}

The main result in this section is the following Klyachko's multiplicity formula for signed $p$-Kostka numbers. The formula basically says that the computation of signed $p$-Kostka numbers can be reduced to the computation of $p$-Kostka numbers and signed $p$-Kostka numbers of the form $k_{(\alpha|\beta),(\lambda|\varnothing)}$ where $\lambda$ is $p$-restricted, i.e., the multiplicity of each projective indecomposable module as a summand of $M(\alpha|\beta)$.
We will explicitly perform the computation of a specific signed $p$-Kostka number in Example \ref{ex:3.4} below. This will show how to concretely apply Theorem \ref{T: Kostka reduction}.

\begin{thm}\label{T: Kostka reduction} Let $(\alpha|\beta),(\lambda|p\mu)\in\P^2(n)$. We have \[k_{(\alpha|\beta),(\lambda|p\mu)}=\sum_{(\b{\gamma}|\b{\delta})\in \suppL{(\alpha|\beta),(\lambda|p\mu)}}k_{(\b{\gamma}^{(0)}|\b{\delta}^{(0)}),(\lambda(0)|\varnothing)} \prod_{i=1}^\infty k_{\b{\gamma}^{(i)},\lambda(i)}k_{\b{\delta}^{(i)},\mu(i-1)}.\]
\end{thm}
\begin{proof} Let $\rho=\Rho(\lambda|p\mu)$. By Theorem \ref{BC1}(ii), since $P_\rho$ is a Green vertex of $Y(\lambda|p\mu)$ by Theorem~\ref{T: Young Broue cor}(i), we have $[M(\alpha|\beta):Y(\lambda|p\mu)]=[M(\alpha|\beta)(P_\rho):Y(\lambda|p\mu)(P_\rho)]$ and therefore by Proposition~\ref{P: Brauer M} and Theorem~\ref{T: Young Broue cor}(ii), we have
\begin{align}\label{Eq: 2} [M(\alpha|\beta):Y(\lambda|p\mu)]=\sum_{(\b{\gamma}|\b{\delta})\in \Lambda((\alpha|\beta),\rho)}\prod_{i=0}^\infty \big(\G_{p^i}(M^{\b{\gamma}^{(i)}},M^{\b{\delta}^{(i)}}):\G_{p^i}(Y^{\lambda(i)},Y^{\mu(i-1)})\big).
\end{align} Fix $(\b{\gamma}|\b{\delta})\in \Lambda((\alpha|\beta),\rho)$. When $i=0$, for \[\big[\G_1(M^{\b{\gamma}^{(0)}},M^{\b{\delta}^{(0)}}):\G_1(Y^{\lambda(0)},Y^\varnothing)\big]=[M(\b{\gamma}^{(0)}|\b{\delta}^{(0)}):Y^{\lambda(0)}]= k_{(\b{\gamma}^{(0)}|\b{\delta}^{(0)}),(\lambda(0)|\varnothing)} \neq 0,\] it is necessary that $(\lambda(0)|\varnothing)\unrhd  (\wp(\b{\gamma}^{(0)})|\wp(\b{\delta}^{(0)}))$. When $i\geq 1$, by Equation \ref{Eq: 1}, Lemmas~\ref{L:trivialnormal} and \ref{L: module V,S}, we have
\begin{align*}
&\big[\G_{p^i}(M^{\b{\gamma}^{(i)}},M^{\b{\delta}^{(i)}}):\G_{p^i}(Y^{\lambda(i)},Y^{\mu(i-1)})\big]\\
=& \big[\G_2(M^{\b{\gamma}^{(i)}},M^{\b{\delta}^{(i)}}):\G_2(Y^{\lambda(i)},Y^{\mu(i-1)})\big]\\
=&\left \{ \begin{array}{ll} k_{\b{\gamma}^{(i)},\lambda(i)}k_{\b{\delta}^{(i)},\mu(i-1)}&\text{if $|\b{\gamma}^{(i)}|=|\lambda(i)|$ and $|\b{\delta}^{(i)}|=|\mu(i-1)|$,}\\ 0&\text{otherwise.}\end{array}\right .
\end{align*} Furthermore, $k_{\b{\gamma}^{(i)},\lambda(i)}k_{\b{\delta}^{(i)},\mu(i-1)}=0$ unless $\lambda(i)\unrhd \wp(\b{\gamma}^{(i)})$ and  $\mu(i-1)\unrhd \wp(\b{\delta}^{(i)})$. This shows that we only need to consider the subset $\suppL{(\alpha|\beta),(\lambda|p\mu)}$ in Equation \ref{Eq: 2}. The proof is now complete.
\end{proof}

\begin{eg}\label{ex:3.4} Let $p=3$, $(\alpha|\beta)=((1^3)|(6,3,3))$ and $(\lambda|p\mu)=((2^2,1^2)|(6,3))$. In this case, $\lambda(0)=(2^2,1^2)$, $\mu(0)=(2,1)$ and $\lambda(i)=\varnothing=\mu(i)$ for all $i\geq 1$. We leave it to the reader to check that $\suppL{(\alpha|\beta),(\lambda|p\mu)}$ consists of 3 elements $(\b{\gamma}|\b{\delta})$ such that $\b{\gamma}^{(i)}=\varnothing=\b{\delta}^{(i+1)}$ for all $i\geq 1$, $\b{\gamma}^{(0)}=(1^3)$ and the pair $(\b{\delta}^{(0)},\b{\delta}^{(1)})$ belongs to the set $\Delta:=\{((3,0,0),(1^3)),((0,3,0),(2,0,1)),((0,0,3),(2,1,0))\}$. Therefore, using Theorem \ref{T: Kostka reduction} and the matrix $\sgnK_6$ as in Example \ref{Eg: signKostka6}, we have
\begin{align*}
  &\ k_{((1^3)|(6,3,3)),((2^2,1^2)|(6,3))}\\
  =&\ \sum_{(\b{\delta}^{(0)},\b{\delta}^{(1)})\in\Delta} k_{(\b{\gamma}^{(0)}|\b{\delta}^{(0)}),(\lambda(0)|\varnothing)}k_{\b{\delta}^{(1)},\mu(0)}\\
  =&\ k_{((1^3)|(3,0,0)),((2^2,1^2)|\varnothing)}k_{(1^3),(2,1)}+k_{((1^3)|(0,3,0)),((2^2,1^2)|\varnothing)}k_{(2,0,1),(2,1)} +k_{((1^3)|(0,0,3)),((2^2,1^2)|\varnothing)}k_{(2,1,0),(2,1)}\\
  =&\ k_{((1^3)|(3)),((2^2,1^2)|\varnothing)}(k_{(1^3),(2,1)}+2\cdot k_{(2,1),(2,1)})\\
  =&\ 3\times 3=9.
\end{align*}
\end{eg}

We end this subsection with the following remark about Klyachko's multiplicity formula \cite[Corollary 9.2]{Klyachko} and Steinberg Tensor Product Theorem.

\begin{cor}[Klyachko's multiplicity formula] Let $\alpha,\lambda\in\P(n)$. We have \[k_{\alpha,\lambda}=\sum \prod_{i=0}^\infty k_{\b{\gamma}^{(i)},\lambda(i)}\] where the sum is taken over all $\b{\gamma}^{(0)},\b{\gamma}^{(1)},\ldots\in \Seq$ such that $\alpha=\sum_{i=0}^\infty p^i\cdot\b{\gamma}^{(i)}$ and $\lambda(i)\unrhd \wp(\b{\gamma}^{(i)})$ for all $i\in\NN_0$.
\end{cor}
\begin{proof} Take $\beta=\varnothing=\mu$ in Theorem~\ref{T: Kostka reduction}, the condition Definition~\ref{D: suppL}(b) is trivially satisfied and we have only $\lambda(i)\unrhd \wp(\b{\gamma}^{(i)})$ for all $i\in\NN_0$.
\end{proof}

Recall our notation as in Remark \ref{Rem: Schur super}(i).

\begin{cor}[Steinberg Tensor Product Theorem]\label{C: steinberg} Let $(\lambda|p\mu)\in\P^2_p(n)$ and $r$ be the maximal such that either $\lambda(r)\neq \varnothing$ or $\mu(r-1)\neq \varnothing$. Then \[L(\lambda|p\mu)\cong L(\lambda(0)|\varnothing)\otimes L(p\lambda(1)|p\mu(0))\otimes \cdots\otimes L(p^r\lambda(r)|p^r\mu(r-1)).\]
\end{cor}
\begin{proof} Let $N=L(\lambda(0)|\varnothing)\otimes L(p\lambda(1)|p\mu(0))\otimes \cdots\otimes L(p^r\lambda(r)|p^r\mu(r-1))$ and $\rho=\Rho(\lambda|p\mu)$. Let $\mathrm{GL}(m|n)(F)$ denote the super general linear group as defined in \cite[Section 2]{BrundanKujawa}. The Frobenius map induces a map $\Frob:\mathrm{GL}(m|n)(F)\to \mathrm{GL}(m)(F)\times \mathrm{GL}(n)(F)$. By \cite[Remark 4.6(iii)]{BrundanKujawa}, $\Frob^* (L(\sigma)\boxtimes L(\tau))\cong L(p\sigma|p\tau)$ is the irreducible $S_F(m|n,s)$-module obtained by inflating through $\Frob$ where $L(\sigma),L(\tau)$ are the irreducible $S_F(m,s)$- and $S_F(n,s)$-modules of highest weights $\sigma,\tau$ respectively. Since $L(\sigma)\boxtimes L(\tau)$ has weights of the form $(\gamma|\delta)$, $\Frob^*(L(\sigma)\boxtimes L(\tau))$ has weights of the form $(p\gamma|p\delta)$. So, for each $i\in\{1,\ldots,r\}$, $L(p^i\lambda(i)|p^i\mu(i-1))$ has weights of the form $(p^i\b{\gamma}^{(i)}|p^i\b{\delta}^{(i)})$ of which the highest weight is $(p^i\lambda(i)|p^i\mu(i-1))$. Furthermore, \begin{align*}
\dim_FL(p^i\lambda(i)|p^i\mu(i-1))^{(p^i\b{\gamma}^{(i)}|p^i\b{\delta}^{(i)})}&=\dim_FL(p^i\lambda(i))^{p^i\b{\gamma}^{(i)}}\dim_FL(p^i\mu(i-1))^{p^i\b{\delta}^{(i)}}\\
&= k_{\b{\gamma}^{(i)},\lambda(i)}k_{\b{\delta}^{(i)},\mu(i-1)}.
\end{align*} When $i=0$, $\dim_FL(\lambda(0)|\varnothing)^{(\b{\gamma}^{(0)}|\b{\delta}^{(0)})}=k_{(\b{\gamma}^{(0)}|\b{\delta}^{(0)}),(\lambda(0)|\varnothing)}$. Therefore, for each $(\alpha|\beta)\in\C^2(n)$, by Theorem \ref{T: Kostka reduction}, we have
\begin{align*}
\dim_F N^{\alpha|\beta}&=\sum_{(\b{\gamma}|\b{\delta})\in \Lambda((\alpha|\beta),\rho)}\prod^r_{i=0}\dim_FL(p^i\lambda(i)|p^i\mu(i-1))^{(p^i\b{\gamma}^{(i)}|p^i\b{\delta}^{(i)})}\\
&=\sum_{(\b{\gamma}|\b{\delta})\in \Lambda((\alpha|\beta),\rho)} k_{(\b{\gamma}^{(0)}|\b{\delta}^{(0)}),(\lambda(0)|\varnothing)}k_{\b{\gamma}^{(1)},\lambda(1)}k_{\b{\delta}^{(1)},\mu(0)}\cdots k_{\b{\gamma}^{(r)},\lambda(r)}k_{\b{\delta}^{(r)},\mu(r-1)}\\
&=\sum_{(\b{\gamma}|\b{\delta})\in \suppL{(\alpha|\beta),(\lambda|p\mu)}} k_{(\b{\gamma}^{(0)}|\b{\delta}^{(0)}),(\lambda(0)|\varnothing)}k_{\b{\gamma}^{(1)},\lambda(1)}k_{\b{\delta}^{(1)},\mu(0)}\cdots k_{\b{\gamma}^{(r)},\lambda(r)}k_{\b{\delta}^{(r)},\mu(r-1)}\\
&=k_{(\alpha|\beta),(\lambda|p\mu)}=\dim_F L(\lambda|p\mu)^{\alpha|\beta}
\end{align*} Since both $N$ and $L(\lambda|p\mu)$ have highest weight $(\lambda|p\mu)$ and the dimensions of their $(\alpha|\beta)$-weight spaces are identical, we conclude that $N\cong L(\lambda|p\mu)$.
\end{proof}

\section{The signed $p$-Kostka numbers $k_{(\alpha|\beta),(\lambda|p\mu)}$ when $|\beta|=p|\mu|$}\label{S: equal case}
In this section, we examine the signed $p$-Kostka numbers $k_{(\alpha|\beta),(\lambda|p\mu)}$ when $|\beta|=p|\mu|$.
The main result is Theorem \ref{T: signedRowRemoval} which relates such signed $p$-Kostka numbers with some $p$-Kostka numbers of some symmetric group algebras of smaller sizes.
We use it to deduce Corollaries \ref{C: signedrowremoval} and \ref{C: principal block}.
In these corollaries we show that the signed $p$-Kostka matrix $\sgnK_n$ has diagonal blocks given by the Kronecker products of some $p$-Kostka matrices and we examine the principal-block summands of the signed Young permutation $F\sym{mp}$-modules where $m\in\NN_0$.

The proof of Theorem \ref{T: signedRowRemoval} built on our Theorem \ref{T: Kostka reduction}. In view of that, we shall look into a few properties of the subset $\suppL{(\alpha|\beta),(\lambda|p\mu)}$, as defined in Definition \ref{D: suppL}, which we shall need later in this paper.
We begin with an easy observation.

\begin{lem}\label{L:first part null} Let $(\alpha|\beta)\in\C^2(n)$, $(\lambda|p\mu)\in\P^2_p(n)$ and $\Lambda^\circ=\suppL{(\alpha|\beta),(\lambda|p\mu)}$.
\begin{enumerate}[(i)]
  \item If $\lambda(0)=\varnothing$ then $\b{\delta}^{(0)}=\varnothing=\b{\gamma}^{(0)}$ for all $(\b{\gamma}|\b{\delta})\in\Lambda^\circ$. The converse holds true if $\Lambda^\circ\neq \emptyset$.
  \item If $|\beta|=p|\mu|$ then $\b{\delta}^{(0)}=\varnothing$ for all $(\b{\gamma}|\b{\delta})\in\Lambda^\circ$ and the map \[\phi:\suppL{(\alpha|\beta),(\lambda|p\mu)}\to \suppL{(\alpha|\varnothing),(\lambda|\varnothing)}\times \suppL{(\beta|\varnothing),(p\mu|\varnothing)}\] given by $\phi((\b{\gamma}|\b{\delta}))=((\b{\gamma}|\pmb{\varnothing}),(\b{\delta}|\pmb{\varnothing}))$ is bijective. Furthermore, if $\b{\delta}^{(0)}=\varnothing$ for all $(\b{\gamma}|\b{\delta})\in\Lambda^\circ$ and $\Lambda^\circ\neq \emptyset$ then $|\beta|=p|\mu|$.
\end{enumerate}
\end{lem}
\begin{proof} Most parts of parts (i) and (ii) follow from Definition \ref{D: suppL}. The fact that $\phi$ is a bijection follows easily by noting that $\lambda(i)\unrhd \wp(\b{\gamma}^{(i)})$ and $(p\mu)(i)=\mu(i-1)\unrhd \wp(\b{\delta}^{(i)})$ for all $i\in \NN_0$.
\end{proof}

We have the following immediate proposition. This serves as a motivation why we are interested in the case when $|\beta|=p|\mu|$.

\begin{prop}\label{P: Kostka eql zero} Let $(\alpha|\beta),(\lambda|p\mu)\in\P^2(mp)$ for some nonnegative integer $m$. If $\lambda(0)=\varnothing$, i.e., $Y(\lambda|p\mu)$ lies in the principal block of $\sym{mp}$, then $k_{(\alpha|\beta),(\lambda|p\mu)}=0$ unless $|\beta|=p|\mu|$.
\end{prop}
\begin{proof} By \cite[Corollary 5.2.9]{HemmerKujawaNakano}, the signed Young module $Y(\lambda|p\mu)$ belongs to the block labelled by the $p$-core of $\lambda$. Therefore $\lambda(0)=\varnothing$ is equivalent to $Y(\lambda|p\mu)$ lies in the principal block of $F\sym{mp}$. By Lemma \ref{L:first part null}(i) and (ii), we see that $\suppL{(\alpha|\beta),(\lambda|p\mu)}=\emptyset$ unless $|\beta|=p|\mu|$. This translates to our result using Theorem \ref{T: Kostka reduction}.
\end{proof}

We are now ready to state the main of this section. It generalizes Theorem \ref{T: RowRemoval} to the signed $p$-Kostka numbers. We refer the reader to Section 2.1 for the combinatorics of row cuts of partitions. This is crucial to state the following result.

\begin{thm}\label{T: signedRowRemoval} Let $(\alpha|\beta),(\lambda|p\mu)\in\P^2(n)$ such that $|\beta|=p|\mu|$ and the pairs $(\alpha,\lambda)$ and $(\beta,p\mu)$ admit horizontal $r$- and $s$-row cuts respectively. Then \[k_{(\alpha|\beta),(\lambda|p\mu)}=k_{\trimt{\alpha}{r},\trimt{\lambda}{r}}k_{\trimt{\beta}{s},\trimt{p\mu}{s}} k_{\trimb{\alpha}{r},\trimb{\lambda}{r}} k_{\trimb{\beta}{s},\trimb{p\mu}{s}}.\]
\end{thm}
\begin{proof} By Lemma \ref{L:first part null}(ii), Theorems \ref{T: RowRemoval} and \ref{T: Kostka reduction}, we have
\begin{align*}
k_{(\alpha|\beta),(\lambda|p\mu)}&=\sum_{(\b{\gamma}|\b{\delta})\in \suppL{(\alpha|\beta),(\lambda|p\mu)}}k_{\b{\gamma}^{(0)},\lambda(0)}\prod_{i=1}^\infty k_{\b{\gamma}^{(i)},\lambda(i)}k_{\b{\delta}^{(i)},\mu(i-1)}\\
&=\left (\sum_{(\b{\gamma}|\pmb{\varnothing})\in \suppL{(\alpha|\varnothing),(\lambda|\varnothing)}}\prod_{i=0}^\infty k_{\b{\gamma}^{(i)},\lambda(i)}\right )\left (\sum_{(\b{\delta}|\pmb{\varnothing})\in \suppL{(\beta|\varnothing),(p\mu|\varnothing)}}\prod_{i=0}^\infty k_{\b{\delta}^{(i)},\mu(i-1)}\right )\\
&=k_{\alpha,\lambda} k_{\beta,p\mu}\\
&=k_{\trimt{\alpha}{r},\trimt{\lambda}{r}}k_{\trimt{\beta}{s},\trimt{p\mu}{s}} k_{\trimb{\alpha}{r},\trimb{\lambda}{r}} k_{\trimb{\beta}{s},\trimb{p\mu}{s}}.
\end{align*}
\end{proof}

Recall that the $p$-Kostka and signed $p$-Kostka matrices are denoted by $\K_n$ and $\sgnK_n$ with respect to fixed total orders $\succ_n$ and $\succ$ on $\P(n)$ and $\P^2_p(n)$ respectively such that the restriction of $\succ$ to $\P(n)$ is $\succ_n$ and $\succ_n$ refines the dominant order $\unrhd$. Also, we denote the Kronecker product of two matrices $A,B$ as \[A\otimes B=\begin{pmatrix} A_{11}B&\cdots&A_{1n}B\\ \vdots&&\vdots\\ A_{m1}B&\cdots&A_{mn}B\end{pmatrix}\] if $A$ is an $(m\times n)$-matrix.
Theorem \ref{T: signedRowRemoval} implies the following corollaries.
The first one gives a proof of our Main Result B in the introduction.

\begin{cor}\label{C: signedrowremoval}\
 \begin{enumerate}[(i)]
 \item Let $(\alpha|\beta),(\lambda|p\mu)\in \P^2(n)$ such that $|\beta|=p|\mu|$. We have \[k_{(\alpha|\beta),(\lambda|p\mu)}=k_{\alpha,\lambda}\cdot k_{\beta,p\mu}.\] In this case, the signed $p$-Kostka number $k_{(\alpha|\beta),(\lambda|p\mu)}$ is nonzero if and only if there exists expansions (not necessarily $p$-adic expansions) $\alpha=\sum_{i=0}^\infty p^i\b{\gamma}^{(i)}$ and $\beta=\sum_{i=1}^\infty p^i\b{\delta}^{(i)}$ such that, for all $i\in\NN_0$, $|\lambda(i)|=|\b{\gamma}^{(i)}|, $ $\lambda(i)\unrhd \wp(\b{\gamma}^{(i)})$, $|\mu(i-1)|=|\b{\delta}^{(i)}|$ and $\mu(i-1)\unrhd \wp(\b{\delta}^{(i)})$.
 \item Let $(\alpha|\beta),(\lambda|p\mu)\in \P^2_p(n)$ such that $|\alpha|=|\lambda|-|\lambda(0)|$. We have \[k_{(\alpha|\beta),(\lambda|p\mu)}=k_{\alpha,\lambda-\lambda(0)}\cdot k_{\beta,\m(\lambda(0))+p\mu}.\]
 \item The signed $p$-Kostka matrix $\sgnK_n$ is lower unitriangular of the form \[\begin{pmatrix} \K_n\\ *&\K_{n-p}\\ *&*&\K_2\otimes \K_{n-2p}\\ *&*&*&\K_3\otimes\K_{n-3p}\\ \vdots&\vdots&\vdots&\vdots&\ddots\\ *&*&*&*&\cdots& \K_d\otimes \K_{n-dp}\end{pmatrix}\] where $0\leq n-dp\leq p-1$ and, for each $0\leq s\leq d$, $\K_s\otimes \K_{n-sp}$ is the Kronecker product of the $p$-Kostka matrices $\K_s$ and $\K_{n-sp}$.
 \end{enumerate}
\end{cor}
\begin{proof} For part (i), the equality is obtained by taking $r=0=s$ in Theorem \ref{T: signedRowRemoval}. Therefore, $k_{(\alpha|\beta),(\lambda|p\mu)}\neq 0$ if and only if $k_{\alpha,\lambda}\neq 0\neq k_{\beta,p\mu}$ and hence the second assertion follows from \cite[Remark, page 55]{Donkin1} noting that $(p\mu)(i)=\mu(i-1)$.

For part (ii), by \cite[Theorem 3.18]{DL} and part (i), \[k_{(\alpha|\beta),(\lambda|p\mu)}=k_{(\beta|\alpha),(\m(\lambda(0))+p\mu|\lambda-\lambda(0))}= k_{\beta,\m(\lambda(0))+p\mu}k_{\alpha,\lambda-\lambda(0)}.\]

For part (iii), since the total order $\succ$ on $\P^2_p(n)$ refines the dominance order, $\sgnK_n$ is lower unitriangular. Suppose that $(\zeta|p\eta),(\lambda|p\mu)\in\P^2_p(n)$ and $s=|\mu|=|\eta|$. By parts (i) and (ii), we have \[(\sgnK_n)_{(\zeta|p\eta),(\lambda|p\mu)}=k_{(\zeta|p\eta),(\lambda|p\mu)}=k_{p\eta,p\mu}\cdot k_{\zeta,\lambda}= k_{\eta,\mu}\cdot k_{\zeta,\lambda}=(\K_s)_{\eta,\mu}(\K_{n-sp})_{\zeta,\lambda}.\] We have chosen the total order $\succ$ so that the signed $p$-Kostka matrix has the form as in the statement.
\end{proof}

\begin{cor}\label{C: principal block} Let $B_0$ be the principal block of $\sym{mp}$ and $(\alpha|\beta)\in\P^2(mp)$. Then \[M(\alpha|\beta)_{B_0}\cong \bigoplus_{\substack{(\lambda|p\mu)\in\P^2_p(mp),\\ \lambda(0)=\varnothing,\ |\beta|=p|\mu|}}Y(\lambda|p\mu)^{\oplus k_{\alpha,\lambda}k_{\beta,p\mu}}.\]
\end{cor}
\begin{proof} Let $Y(\lambda|p\mu)$ belong to $B_0$, i.e., $\lambda(0)=\varnothing$, and suppose that $k_{(\alpha|\beta),(\lambda|p\mu)}>0$. Our result now follows from Proposition \ref{P: Kostka eql zero} and Corollary \ref{C: signedrowremoval}(i).
\end{proof}

We conclude the section with a short example to illustrate the main results just obtained.

\begin{eg} Let us consider $p=3$ and $n=6$. For $|\alpha|=|\lambda|$ and $p\mu\unrhd p\zeta$, using Corollary \ref{C: signedrowremoval}(i), we have \[k_{(\alpha|p\zeta),(\lambda|p\mu)}=k_{\alpha,\lambda}k_{p\zeta,p\mu}=k_{\alpha,\lambda}\] since, for example, $k_{(3,3),(6)}=1$. Therefore the three diagonal blocks in Example \ref{Eg: signKostka6} are now fully explained and they are precisely $\K_6$, $\K_3$ and $\K_2\otimes \K_0$ as in Corollary \ref{C: signedrowremoval}(iii).

Next we also explain why \[k_{(\alpha|(3)),(\lambda|\varnothing)}=0\] for any $\alpha,\lambda$ such that $3=|\alpha|=|\lambda|-|\lambda(0)|$. Using Corollary \ref{C: signedrowremoval}(ii), we have \[k_{(\alpha|(3)),(\lambda|\varnothing)}=k_{\alpha,(3)}k_{(3),\m(\lambda(0))}.\] But $\m(\lambda(0))$ is necessarily $p$-restricted (it is the label of the simple module $D_{\lambda(0)}\otimes\sgn$) and therefore $\m(\lambda(0))\neq (3)$, which implies that $k_{(3),\m(\lambda(0))}=0$.

Next we consider the principal block $B_0$ of $F\sym{6}$. The signed Young modules belonging to $B_0$ are labelled by $((6)|\varnothing)$, $((3^2)|\varnothing)$, $((3)|(3))$, $(\varnothing|(6))$ and $(\varnothing|(3^2))$. Corollary \ref{C: principal block} shows that the entries of the columns below the diagonal blocks and labelled by the bipartitions labelling $B_0$ are all equal to zero. For an interesting example, $k_{(\alpha|(3)),((6)|\varnothing)}=0=k_{(\alpha|(3)),((3^2)|\varnothing)}$ where $|\alpha|=3$, i.e., \[M(\alpha|(3))_{B_0}\cong\left \{\begin{array}{ll} Y((3)|(3)) &\text{if $\alpha=(3)$,}\\ 0&\text{if $\alpha=(2,1),(1^3)$.}\end{array}\right .\] In general, in $F\sym{2p}$, $M(\alpha|(p))_{B_0}$ is isomorphic to $Y((p)|(p))$ if $\alpha=(p)$ and 0 otherwise.
\end{eg}

\section{The signed $p$-Kostka numbers $k_{(\alpha|\beta),(\lambda|p\mu)}$ in the general case}\label{S: not equal case}

In this section, we examine the validity of Theorem \ref{T: signedRowRemoval} without the condition $|\beta|=p|\mu|$. This general case is slightly more complicated and we need a separate treatment. More precisely, we need to examine the Specht series for skew representations introduced by James-Peel in \cite{JamesPeel}.

We now describe the minimal ingredient from \cite{JamesPeel} that we shall require in the proof of Lemma \ref{L: Specht submodule}. A diagram is a finite subset of $\Z\times\Z$. Let $\alpha$ and $\beta$ be partitions such that $\alpha_i\geq \beta_i$ for all $i\in\NN$. We have the skew diagram \[D=[\alpha\backslash\beta]=\{(i,j)\in\NN\times\NN:1\leq i\leq\ell(\alpha),\ \beta_i<j\leq \alpha_i\}.\] Let $-D$ (or $-\alpha\backslash\beta$) be the diagram $\{(i,-j):(i,j)\in D\}$ and let $n=|\alpha|-|\beta|$. There is a $F\sym{n}$-module $S^{\alpha\backslash\beta}$ (or $S^D$) labelled by the diagram $D$ called a skew representation. In the case when $\beta=({\alpha_{s+1}}^s)$ for some $s\in\NN_0$, we have that $S^{\alpha\backslash\beta}$ is isomorphic to $\ind_{\sym{d}\times\sym{n-d}}^{\sym{n}}(S^{\trimt{\alpha}{s}-\beta}\boxtimes S^{\trimb{\alpha}{s}})$ where $d=|\trimt{\alpha}{s}|-|\beta|$ (see \cite[Page 345]{JamesPeel}). There is a unique node $(r,s)\in D$ such that, if $(i,j)\in D$, then $i\geq r$ and $j\leq s$. A node $(i,j)\in D$ is deposited in $D$ if $(i',j')\in D$ whenever $r\leq i'\leq i$ and $j\leq j'\leq s$. All the operations in the next paragraph fix the node $(r,s)$.

Let $D$ be a skew diagram and $(r,s)\in D$ be the unique node as in the previous paragraph. The operation $X$ is defined so that $D^X$ is a diagram obtained from $D$ by permuting the rows or columns of $D$. In this case, $S^D\cong S^{D^X}$ (see \cite[2.1]{JamesPeel}). Let $(i_1,j_1),(i_2,j_2)\in D$ but $(i_1,j_2)\not\in D$ and $(i_2,j_1)\not\in D$. The diagram $D^I$ (respectively, $D^K$), with respect to the nodes $(i_1,j_1)$ and $(i_2,j_2)$, is defined as the diagram obtained from $D$ by moving the nodes $(i_1,j)\in D$ such that $(i_2,j)\not\in D$ to $(i_2,j)$ (respectively, $(i,j_1)\in D$ such that $(i,j_2\not\in D$ to $(i,j_2)$)and keeping the remaining nodes; namely, $D^I$ and $D^K$ are obtained from $D$ by moving all the nodes in the $i_1$th row (respectively, $j_1$th column) to the respective vacant positions in $i_2$th row (respectively, $j_2$th column). Suppose there exists $i^*>r$ such that $(i^*,j)\not\in D$ for any $j\in \Z$ but there exists a node $(i_1,j_1)\in D$ for some $i_1>i^*$ or there exists $j^*<s$ such that $(i,j^*)\not\in D$ for any $i\in \Z$ but there exists a node $(i_1,j_1)\in D$ for some $j_1<j^*$, i.e., we have the condition \cite[Theorem 4.7 (ii)--(iii)]{JamesPeel}. Let $Y$ be the operation so that $D^Y$ is the diagram obtained from $D$ by exchanging the $(i+1)$th row (respectively, $(j+1)$th column) of with the $i$th row (respectively, $j$th column) successively in the order of and for each $i=i^*,i^*+1,\ldots$ (respectively, $j=j^*,j^*+1,\ldots$). Literally, $D^Y$ is obtained from $D$ by deleting the $i^*$th row (respectively, $j^*$th column). In this case, $S^D\cong S^{D^Y}$ (see \cite[Lemma 4.8]{JamesPeel}).

\begin{lem}\label{L: Specht submodule} Let $\lambda\in\P(n)$, $\lambda=\xi\cont\zeta$ for some partitions $\xi,\zeta$ and $d=|\xi|$. Then
\begin{enumerate}[(i)]
  \item $S^\lambda$ is isomorphic to a submodule of the induced module $\ind_{\sym{d}\times\sym{n-d}}^{\sym{n}}(S^\xi\boxtimes S^\zeta)$, and
  \item $\big[\ind_{\sym{d}\times\sym{n-d}}^{\sym{n}}(Y^\xi\boxtimes Y^\zeta):Y^{\lambda}\big]=1$.
\end{enumerate}
\end{lem}
\begin{proof} We use the notation as in the paragraphs prior to the statement. Let $k=\ell(\xi)$, $\beta=({\zeta_1}^k)$, $\alpha=(\xi+\beta)\cont\zeta$ and let $D$ be the diagram of $\alpha\backslash\beta$. Then $S^{\alpha\backslash\beta}$ is a skew representation isomorphic to $\ind_{\sym{d}\times\sym{n-d}}^{\sym{n}}(S^\xi\boxtimes S^\zeta)$. By \cite[Theorems 4.15 and 4.16]{JamesPeel}, it suffices to show that we get the diagram $-[\lambda]$ by successive application of the operations either $X$, $Y$ or $K$ (and without $I$). We refer to Figure A for the following explanation. The red dot denotes the unique node $(r,s)$ in each diagram. The pair of blue dots in each diagram denotes the pair satisfying the operation $K$. We first apply the operation $X$ where $-\xi$ and $-\zeta$ denote the mirror images of $\xi$ and $\zeta$ respectively in the obvious way. All the nodes in $-\xi$ are deposited. Applying the process $K$, the most right column of $-\zeta$ moves to the bottom of the most right column of $-\xi$. We now apply the operation $Y$ to rearrange the columns, i.e., deleting the empty column. Now the pair of blue nodes satisfies the operation $K$. We can repeat the arguments for $\zeta_1$ times, since, by assumption $\xi_k\geq \zeta_1$, and get the diagram $-\lambda$. Since we applied the operations $X$, $Y$ or $K$ each time, the Specht module $S^\lambda$ appears as a submodule in some Specht series of $\ind_{\sym{d}\times\sym{n-d}}^{\sym{n}}(S^\xi\boxtimes S^\zeta)$. The proof of part (i) is now complete.
\[\begin{tikzpicture}[scale=.4]
\node at (14,10) {Figure A};
\node at (-1.1,4) {$D=$};
\node at (5,6) {$\xi$};
\node at (1,3) {$\zeta$};
\draw (0.5,0)--(1.5,0)--(1.5,2)--(2.5,2)--(2.5,3)--(3,3)--(3,4)--(0.5,4)--(0.5,0);
\draw (3,4)--(6,4)--(6,6)--(7,6)--(7,7)--(9,7)--(9,8)--(3,8)--(3,4);
\node at (10.5,4) {$\cong$};
\node at (10.5,5) {$X$};
\draw[gray!20,fill=gray!20] (13.5,0) rectangle (14,4);
\draw (14,0)--(13,0)--(13,2)--(12,2)--(12,3)--(11.5,3)--(11.5,4)--(14,4)--(14,0);
\draw (20,4)--(17,4)--(17,6)--(16,6)--(16,7)--(14,7)--(14,8)--(20,8)--(20,4);
\draw[dashed] (17,4)--(14,4)--(14,7);
\node at (18,6) {$-\xi$};
\node at (13,3) {$-\zeta$};
\draw[red,fill=red] (8.75,7.75) circle (.1cm);
\draw[red,fill=red] (19.75,7.75) circle (.1cm);
\draw[blue,fill=blue] (19.75,4.2) circle (.1cm);
\draw[blue,fill=blue] (13.75,3.75) circle (.1cm);
\draw[dashed] (19.5,8)--(19.5,4);
\draw [->,thick] (11,-0.5)--(8,-2);
\node at (9.5,-.5) {$K$};
\draw[gray!20,fill=gray!20] (8.5,-11) rectangle (9,-7);
\draw (2.5,-11)--(2,-11)--(2,-9)--(1,-9)--(1,-8)--(0.5,-8)--(0.5,-7)--(2.5,-7)--(2.5,-11);
\draw (8.5,-11)--(8.5,-7)--(6,-7)--(6,-5)--(5,-5)--(5,-4)--(3,-4)--(3,-3)--(9,-3)--(9,-11)--(9,-11)--(8.5,-11);
\draw[dashed] (2.5,-7)--(2.5,-3);
\draw[dashed] (3,-11)--(3,-3);
\draw[dashed] (2.5,-7)--(6,-7);
\draw[dashed] (8.5,-3)--(8.5,-7);
\draw[red,fill=red] (8.75,-3.25) circle (.1cm);
\node at (10.5,-7) {$\cong$};
\node at (10.5,-6) {$Y$};
\draw[red,fill=red] (19.75,-3.25) circle (.1cm);
\draw[blue,fill=blue] (19.25,-6.75) circle (.1cm);
\draw[gray!20,fill=gray!20] (13.5,-11) rectangle (14,-7);
\draw[blue,fill=blue] (13.75,-7.25) circle (.1cm);
\draw (14,-11)--(13.5,-11)--(13.5,-9)--(12.5,-9)--(12.5,-8)--(12,-8)--(12,-7)--(14,-7)--(14,-11);
\draw (19.5,-11)--(19.5,-7)--(17,-7)--(17,-5)--(16,-5)--(16,-4)--(14,-4)--(14,-3)--(20,-3)--(20,-11)--(19.5,-11);
\draw[dashed] (14,-4)--(14,-7)--(17,-7);
\draw[dashed] (19.5,-3)--(19.5,-7);
\draw[dashed] (19,-3)--(19,-7);
\draw (32,-11)--(32,-9)--(31,-9)--(31,-8)--(30.5,-8)--(30.5,-7) --(30,-7)--(30,-5)--(29,-5)--(29,-4)--(27,-4)--(27,-3)--(33,-3)--(33,-11)--(32,-11);
\draw [->,thick,dashed] (22,-7)--(25,-7);
\node at (31.5,-6) {$-\lambda$};
\node at (23.5,-6) {\small{only $Y$ and $K$}};
\end{tikzpicture}\]

For part (ii), recall that, for any partition $\mu$, by the Submodule Theorem (see \cite[4.8]{James}), the Young module $Y^\mu$ is the unique summand of the permutation module $M^\mu$ containing the Specht module $S^\mu$ up to isomorphism. Up to isomorphism, by part (i), since the induction functor is exact, we have \[S^\lambda\subseteq \ind_{\sym{d}\times\sym{n-d}}^{\sym{n}}(S^\xi\boxtimes S^\zeta)\subseteq \ind_{\sym{d}\times\sym{n-d}}^{\sym{n}}(Y^\xi\boxtimes Y^\zeta)\mid \ind_{\sym{d}\times\sym{n-d}}^{\sym{n}}(M^\xi\boxtimes M^\zeta)=M^\lambda.\] Therefore, we conclude that \[1\leq \big[\ind_{\sym{d}\times\sym{n-d}}^{\sym{n}}(Y^\xi\boxtimes Y^\zeta):Y^{\lambda}\big]\leq \big[M^{\lambda}:Y^{\lambda}\big]=1.\]
\end{proof}

\begin{cor}\label{C: ineq 0 case} Let $(\gamma|\delta),(\lambda|\varnothing)\in\P^2(n)$ and $\gamma',\gamma'',\delta',\delta''$ be partitions such that $\gamma=\gamma'\cup\gamma''$, $\delta=\delta'\cup\delta''$, $|\trimt{\lambda}{r}|=|\gamma'|+|\delta'|$ and $|\trimb{\lambda}{r}|=|\gamma''|+|\delta''|$ for some nonnegative integer $r$. Then \[k_{(\gamma|\delta),(\lambda|\varnothing)}\geq k_{(\gamma'|\delta'),(\trimt{\lambda}{r}|\varnothing)}k_{(\gamma''|\delta''),(\trimb{\lambda}{r}|\varnothing)}.\]
\end{cor}
\begin{proof} Let $a=k_{(\gamma'|\delta'),(\trimt{\lambda}{r}|\varnothing)}$ and $b=k_{(\gamma''|\delta''),(\trimb{\lambda}{r}|\varnothing)}$. Since $M(\gamma|\delta)\cong \ind^{\sym{n}}_{\sym{d}\times\sym{n-d}}(M(\gamma'|\delta')\boxtimes M(\gamma''|\delta''))$ where $d=|\trimt{\lambda}{r}|$, we have \[\big (\ind^{\sym{n}}_{\sym{d}\times\sym{n-d}}(Y^{\trimt{\lambda}{r}}\boxtimes Y^{\trimb{\lambda}{r}})\big )^{\oplus ab}\mid M(\gamma|\delta).\] By Lemma \ref{L: Specht submodule}(ii), we have \[\big [M(\gamma|\delta):Y^\lambda\big]\geq ab\cdot \big [\ind^{\sym{n}}_{\sym{d}\times\sym{n-d}}(Y^{\trimt{\lambda}{r}}\boxtimes Y^{\trimb{\lambda}{r}}):Y^\lambda\big]=ab.\]
\end{proof}

Corollary \ref{C: ineq 0 case} is a special case of our main result Theorem \ref{T: signedRowRemoval2}. To prove the general version Theorem \ref{T: signedRowRemoval2}, we require the following key lemmas. We remind the reader the notation we have introduced in Section \ref{noth Sn}.

\begin{lem}\label{L: padic top} Let $\lambda$ be a partition, $r\in\NN_0$, $b_i=\lambda(i)_{r+1}$ for each nonnegative integer $i$ and $b=\sum_{i=0}^\infty p^ib_i=\lambda_{r+1}$. Then, for any nonnegative integer $i$, we have both \begin{align*}
  \trimb{\lambda}{r}(i)&=\trimb{\lambda(i)}{r},\\
  \big(\trimt{\lambda}{r}-(b^r)\big)(i)&=\trimt{\lambda(i)}{r}-\big ({b_i}^r\big).
\end{align*}
\end{lem}
\begin{proof} Since $\lambda(i)$ is a $p$-restricted partition, it is clear that $\trimb{\lambda(i)}{r}$ is also a $p$-restricted partition. Observe that $\trimb{\lambda}{r}=\sum_{i=0}^\infty p^i\trimb{\lambda(i)}{r}$. This shows the first equality. For the second equality, first observe that \[\sum^\infty_{i=0}p^i\big (\trimt{\lambda(i)}{r}-({b_i}^r)\big )=\sum^\infty_{i=0} \big (p^i\trimt{\lambda(i)}{r}-((p^ib_i)^r)\big )=\trimt{\lambda}{r}-(b^r).\] It remains to show that, for each $i$, $\gamma:=\trimt{\lambda(i)}{r}-\big ((b_i)^r\big)$ is a $p$-restricted partition. Notice that $\gamma$ is a $p$-restricted partition because $\gamma_{r}=\trimt{\lambda(i)}{r}_r-b_i=\trimt{\lambda(i)}{r}_{r}-\trimt{\lambda(i)}{r}_{r+1}$ and, for each $1\leq j\leq r-1$, we have \[\gamma_j-\gamma_{j+1}=(\trimt{\lambda(i)}{r}_j-b_i)-(\trimt{\lambda(i)}{r}_{j+1}-b_i)= \trimt{\lambda(i)}{r}_j-\trimt{\lambda(i)}{r}_{j+1}.\]
\end{proof}


\begin{lem}\label{L: dominant block} Let $\lambda\in\P(n)$, $k\in\NN_0$ and $\gamma\in\C(n)$ such that $\ell(\gamma)\leq k$. If $\lambda\unrhd\wp(\gamma)$ then $\gamma_j\geq \lambda_k$ for all $1\leq j\leq k$.
\end{lem}
\begin{proof} By assumption, $s:=\ell(\wp(\gamma))\leq k$. Let $\wp(\gamma)=(a_1,\ldots,a_k)$ (assume that $a_i=0$ if $i>s$) and suppose on the contrary that $a_i<\lambda_k$ for some $1\leq i\leq k$. Then \[n-(k-i+1)\lambda_k<n-\sum^k_{j=i}a_j=\sum^{i-1}_{j=1}a_j\leq \sum^{i-1}_{j=1}\lambda_j\leq n-\sum^k_{j=i}\lambda_j\leq n-(k-i+1)\lambda_k.\] Alternatively, one observes that the Young diagram of $\wp(\gamma)$ is obtained from $\lambda$ by successively moving a node from a higher row to a lower row such that, at each step, it remains as a partition. Therefore, since we cannot move a node lower than row $k$, the block $((\lambda_k)^k)$ in $\lambda$ will remain in each step and therefore in $\wp(\gamma)$.
\end{proof}

\begin{lem}\label{L: injective map} Let $(\alpha|\beta),(\lambda|p\mu)\in\P^2(n)$ such that the pairs $(\alpha,\lambda)$ and $(\beta,p\mu)$ admit horizontal $r$- and $s$-row cuts respectively. For each $i\in\NN_0$, let $b_i=\lambda(i)_{r+1}$, $c_i=\mu(i)_{s+1}$, $b=\sum^\infty_{i=0}b_i=\lambda_{r+1}$ and $c=\sum^\infty_{i=0}p^{i+1}c_i=p\mu_{s+1}$.
\begin{enumerate}[(i)]
  \item We have that both $\trimt{\alpha}{r}-(b^r)$ and $\trimt{\beta}{s}-(c^s)$ are partitions.
  \item Let
  \begin{align*}
    \Gamma_1&=\suppL{\trimt{\alpha}{r}-(b^r),\trimt{\lambda}{r}-(b^r)},& \Gamma_2&=\suppL{\trimt{\beta}{s}-(c^s),\trimt{p\mu}{s}-(c^s)},\\
    \Gamma_3&=\suppL{(\trimb{\alpha}{r}|\trimb{\beta}{s}),(\trimb{\lambda}{r}|\trimb{p\mu}{s})},&
    \Gamma_4&=\suppL{(\alpha|\beta),(\lambda|p\mu)},
  \end{align*} $\Gamma=\Gamma_1\times\Gamma_2\times\Gamma_3$, and, for each $\b{\sigma}\in\Gamma_1$, $\b{\tau}\in\Gamma_2$ and $(\b{\gamma}|\b{\delta})\in\Gamma_3$, let $\b{\eta}^{(i)}=(\trant{{\b{\sigma}^{(i)}}}{r}+({b_i}^r))\cont \b{\gamma}^{(i)}$ and $\b{\theta}^{(i)}=(\trant{{\b{\tau}^{(i)}}}{s}+({c_{i-1}}^s))\cont \b{\delta}^{(i)}$ (assume $c_{-1}=0$ so that $\b{\theta}^{(0)}=\trant{{\varnothing}}{s}\cont\b{\delta}^{(0)}$). Then the map $\iota:\Gamma\to\Gamma_4$ defined as $\iota((\b{\sigma},\b{\tau},(\b{\gamma}|\b{\delta})))=(\b{\eta}|\b{\theta})$ is injective.
\end{enumerate}
\end{lem}
\begin{proof} Part (i) follows from Lemma \ref{L: dominant block}. For part (ii), it suffices to check that $(\b{\eta}|\b{\theta})\in\Gamma_4$. Notice that $\b{\sigma}^{(i)}=\trant{\b{\sigma}^{(i)}}{r}\cont\varnothing$ and $\b{\tau}^{(i)}=\trant{\b{\tau}^{(i)}}{s}\cont\varnothing$. We will be using Lemmas \ref{L: padic top} and \ref{L: dominant block} repeatedly. First of all, we have
\begin{align*}
  \sum_{i=0}^\infty p^i\b{\eta}^{(i)}&=\sum_{i=0}^\infty p^i(\trant{{\b{\sigma}^{(i)}}}{r}+({b_i}^r))\cont \sum_{i=0}^\infty p^i\b{\gamma}^{(i)}= (\trimt{\alpha}{r}-(b^r)+(b^r))\cont \trimb{\alpha}{r}=\alpha,\\
  \sum_{i=0}^\infty p^i\b{\theta}^{(i)}&=\sum_{i=0}^\infty p^i(\trant{{\b{\tau}^{(i)}}}{s}+({c_{i-1}}^s))\cont \sum_{i=0}^\infty p^i\b{\delta}^{(i)}=(\trimt{\beta}{s}-(c^s)+(c^s))\cont \trimb{\beta}{s}=\beta,\\
  |\b{\eta}^{(0)}|+|\b{\theta}^{(0)}|&=|\trimt{\lambda(0)}{r}|+|\b{\gamma}^{(0)}|+|\b{\delta}^{(0)}|=|\trimt{\lambda(0)}{r}|+ |\trimb{\lambda}{r}(0)|=|\lambda(0)|,
\end{align*} and, for $i\geq 1$,
\begin{align*}
  |\b{\eta}^{(i)}|&=|(\trant{{\b{\sigma}^{(i)}}}{r}+({b_i}^r))\cont \b{\gamma}^{(i)}|=|(\trimt{\lambda}{r}-(b^r))(i)|+rb_i+|\trimb{\lambda}{r}(i)|\\
  &= |\trimt{\lambda(i)}{r}|-|({b_i}^r)|+rb_i+|\trimb{\lambda(i)}{r}|=|\lambda(i)|,\\
  |\b{\theta}^{(i)}|&=|(\trant{{\b{\tau}^{(i)}}}{s}+({c_{i-1}}^s))\cont \b{\delta}^{(i)}|= |(\trimt{p\mu}{s}-(c^s))(i)|+sc_{i-1}+|(\trimb{p\mu}{s})(i)|\\
  &= |\trimt{(p\mu)(i)}{s}|-|({c_{i-1}}^s)|+sc_{i-1}+|\trimb{(p\mu)(i)}{s}|=|(p\mu)(i)|=|\mu(i-1)|.
\end{align*} Since $b_i=\lambda(i)_{r+1}$, we see that $\wp(\b{\eta}^{(i)})=\wp(\trant{{\b{\sigma}^{(i)}}}{r}+({b_i}^r))\cont\wp(\b{\gamma}^{(i)})$. For all $1\leq k\leq r$ and $\ell>r$, we have
\begin{align*}
  \sum^k_{j=1}\lambda(i)_j&=\sum^k_{j=1}\trimt{\lambda(i)}{r}_j=k b_i+\sum^k_{j=1}(\trimt{\lambda}{r}-(b^r))(i)_j\geq \sum^k_{j=1}\wp(\trant{{\b{\sigma}^{(i)}}}{r}+({b_i}^r))_j=\sum^k_{j=1}\wp(\b{\eta}^{(i)})_j,\\
  \sum^\ell_{j=1}\lambda(i)_j&=|\trimt{\lambda(i)}{r}|+\sum^{\ell-r}_{j=1}\trimb{\lambda(i)}{r}_j= |(\trimt{\lambda}{r}-(b^r))(i)|+rb_i+\sum^{\ell-r}_{j=1}\trimb{\lambda}{r}(i)_j\\
  &=|\trant{{\b{\sigma}^{(i)}}}{r}|+rb_i+\sum^{\ell-r}_{j=1}\trimb{\lambda}{r}(i)_j\geq|\trant{{\b{\sigma}^{(i)}}}{r}|+rb_i+ \sum^{\ell-r}_{j=1}\wp(\b{\gamma}^{(i)})_j=\sum^\ell_{j=1}\wp(\b{\eta}^{(i)})_j.
\end{align*} Therefore $\lambda(i)\unrhd\wp(\b{\eta}^{(i)})$. The proofs of $\lambda(0)\unrhd (\wp(\b{\eta}^{(0)})|\wp(\b{\theta}^{(0)}))$ and, when $i\geq 1$, $\wp(\b{\theta}^{(i)})=\wp(\trant{{\b{\tau}^{(i)}}}{s}+({c_{i-1}}^s))\cont \wp(\b{\delta}^{(i)})$ and $\mu(i-1)\unrhd\wp(\b{\theta}^{(i)})$ are similar.
\end{proof}

We are now ready to state and prove our main result of this section. Observe in the proof that the main obstruction for the inequality in Theorem \ref{T: signedRowRemoval2} to be an equality lies in both Corollary \ref{C: ineq 0 case} and Lemma \ref{L: injective map}. If the inequality in Corollary \ref{C: ineq 0 case} is an equality and the map $\iota$ in Lemma \ref{L: injective map} is bijective then we would have obtained an equality in Theorem \ref{T: signedRowRemoval2}.

\begin{thm}\label{T: signedRowRemoval2} Let $(\alpha|\beta),(\lambda|p\mu)\in\P^2(n)$ such that the pairs $(\alpha,\lambda)$ and $(\beta,p\mu)$ admit horizontal $r$- and $s$-row cuts respectively. Then \[k_{(\alpha|\beta),(\lambda|p\mu)}\geq k_{\trimt{\alpha}{r},\trimt{\lambda}{r}}k_{\trimt{\beta}{s},\trimt{p\mu}{s}} k_{(\trimb{\alpha}{r}|\trimb{\beta}{s}),(\trimb{\lambda}{r}|\trimb{p\mu}{s})}.\]
\end{thm}
\begin{proof} We use the notation as in Lemma \ref{L: injective map} and will be using Theorem \ref{T: Kostka reduction}, \cite[Corollary 1.1]{BowGia}, Corollary \ref{C: ineq 0 case} and Lemmas \ref{L: padic top}, \ref{L: dominant block}, \ref{L: injective map} repeatedly in the calculation. We first notice that $\b{\tau}^{(0)}=\varnothing$ for any $\b{\tau}\in\Gamma_2$.
\begin{align*}
  &k_{\trimt{\alpha}{r},\trimt{\lambda}{r}}k_{\trimt{\beta}{s},\trimt{p\mu}{s}} k_{(\trimb{\alpha}{r}|\trimb{\beta}{s}),(\trimb{\lambda}{r}|\trimb{p\mu}{s})}\\
  =\ &k_{\trimt{\alpha}{r}-(b^r),\trimt{\lambda}{r}-(b^r)}k_{\trimt{\beta}{s}-(c^s),\trimt{p\mu}{s}-(c^s)} \sum_{(\b{\gamma}|\b{\delta})\in \Gamma_3} k_{(\b{\gamma}^{(0)}|\b{\delta}^{(0)}),(\trimb{\lambda}{r}(0)|\varnothing)} \prod_{i=1}^\infty k_{\b{\gamma}^{(i)},\trimb{\lambda}{r}(i)}k_{\b{\delta}^{(i)},\trimb{\mu}{s}(i-1)}\\
  =\ &\sum_{(\b{\sigma},\b{\tau},(\b{\gamma}|\b{\delta}))\in\Gamma}\prod^\infty_{i=0}k_{\b{\sigma}^{(i)},\big (\trimt{\lambda}{r}-(b^r)\big )(i)}k_{\b{\tau}^{(i)},\big(\trimt{p\mu}{s}-(c^s)\big )(i)}k_{(\b{\gamma}^{(0)}|\b{\delta}^{(0)}),(\trimb{\lambda(0)}{r}|\varnothing)} \prod_{i=1}^\infty k_{\b{\gamma}^{(i)},\trimb{\lambda(i)}{r}}k_{\b{\delta}^{(i)},\trimb{\mu(i-1)}{s}}\\
  =\ &\sum_{(\b{\sigma},\b{\tau},(\b{\gamma}|\b{\delta}))\in\Gamma}\prod^\infty_{i=0}k_{\b{\sigma}^{(i)}, \trimt{\lambda(i)}{r}-({b_i}^r)}\prod^\infty_{i=1}k_{\b{\tau}^{(i)},\trimt{\mu(i-1)}{s}-({c_{i-1}}^s)} k_{(\b{\gamma}^{(0)}|\b{\delta}^{(0)}),(\trimb{\lambda(0)}{r}|\varnothing)} \prod_{i=1}^\infty k_{\b{\gamma}^{(i)},\trimb{\lambda(i)}{r}}k_{\b{\delta}^{(i)},\trimb{\mu(i-1)}{s}}\\
  =\ &\sum_{(\b{\sigma},\b{\tau},(\b{\gamma}|\b{\delta}))\in\Gamma}k_{\trant{\b{\sigma}^{(0)}}{r}+({b_0}^r), \trimt{\lambda(0)}{r}} k_{(\b{\gamma}^{(0)}|\b{\delta}^{(0)}),(\trimb{\lambda(0)}{r}|\varnothing)} \prod_{i=1}^\infty k_{\trant{\b{\sigma}^{(i)}}{r}+({b_i}^r), \trimt{\lambda(i)}{r}} k_{\b{\gamma}^{(i)},\trimb{\lambda(i)}{r}}k_{\trant{\b{\tau}^{(i)}}{s}+({c_{i-1}}^s),\trimt{\mu(i-1)}{s}} k_{\b{\delta}^{(i)},\trimb{\mu(i-1)}{s}}\\
  \leq\ &\sum_{(\b{\sigma},\b{\tau},(\b{\gamma}|\b{\delta}))\in\Gamma}k_{((\trant{\b{\sigma}^{(0)}}{r}+({b_0}^r) )\cont \b{\gamma}^{(0)}|\b{\delta}^{(0)}), (\lambda(0)|\varnothing)}\prod_{i=1}^\infty k_{(\trant{\b{\sigma}^{(i)}}{r}+({b_i}^r) )\cont\b{\gamma}^{(i)}, \lambda(i)} k_{(\trant{\b{\tau}^{(i)}}{s}+({c_{i-1}}^s))\cont\b{\delta}^{(i)},\mu(i-1)}\\
  \leq \ &\sum_{(\b{\eta}|\b{\theta})\in\Gamma_4}k_{(\b{\eta}^{(0)}|\b{\theta}^{(0)}),\lambda(0)}\prod^\infty_{i=1} k_{\b{\eta}^{(i)},\lambda(i)}k_{\b{\theta}^{(i)},\mu(i-1)}\\
  =\ &k_{(\alpha|\beta),(\lambda|p\mu)}.
\end{align*}
\end{proof}

\section{Labelling of signed Young permutation modules and mixed powers}\label{S: label of sign}

In this final section, we address the question regarding the labelling of signed Young permutation modules. It is known in the classical case that, if $\lambda,\mu\in\P(n)$, then $M^\lambda\cong M^\mu$ if and only if $\lambda=\mu$ because each Young permutation module $M^\lambda$ has a distinguished indecomposable summand the Young module $Y^\lambda$. In the case of signed Young permutation modules, it is known from the work of Donkin \cite{Do} that, if $(\alpha|p\beta),(\sigma|p\tau)\in\P^2_p(n)$, then $M(\alpha|p\beta)\cong M(\sigma|p\tau)$ if and only if $\alpha=\sigma$ and $\beta=\tau$, again because the signed Young permutation module $M(\alpha|p\beta)$ has a distinguished indecomposable summand the signed Young module $Y(\alpha|p\beta)$.
We refer the reader to Section \ref{SS: sym module} for the combinatorial definitions describing these labelling sets.
However, it is not clear the exact condition for $M(\alpha|\beta)\cong M(\sigma|\tau)$. In this section, we deal with the general case and obtain the following result.

\begin{thm}\label{T: char of M} Let $(\alpha|\beta),(\sigma|\tau)\in\P^2(n)$. Then $M(\alpha|\beta)\cong M(\sigma|\tau)$ if and only if $\alpha=\rho\cont(1^a)$, $\sigma=\rho\cont (1^b)$, $\beta=\zeta\cont(1^c)$ and $\tau=\zeta\cont(1^d)$ for some non-negative integers $a,b,c,d$ and partitions $\rho,\zeta$.
\end{thm}

We shall first recall the definition of semistandard tableau and signed Young Rule.

\begin{defn} Let $\lambda\in\P(n)$ and let $(\alpha|\beta)\in\C^2(n)$. A semistandard $\lambda$-tableau $\t$ of type $(\alpha|\beta)$ is an assignment of the nodes of $[\lambda]$ with colours \[\bb{c}_1<\bb{c}_2<\cdots<\bb{d}_1<\bb{d}_2<\cdots\] such that the following two conditions are satisfied.
\begin{enumerate}[(i)]
  \item The subtableau $\mathfrak{s}$ of $\t$ occupied by the colours $\bb{c}_1,\bb{c}_2,\ldots$ is a row semistandard $\mu$-tableau of type $\alpha$ for some partition $\mu$ where $[\mu]\subseteq [\lambda]$.
  \item The skew $(\lambda/\mu)$-tableau $\t/\mathfrak{s}$ occupied by the colours $\bb{d}_1,\bb{d}_2,\ldots$ is a column semistandard of type $\beta$.
\end{enumerate} The total number of semistandard $\lambda$-tableaux of type $(\alpha|\beta)$ is denoted by $s^\lambda_{\alpha|\beta}$.
\end{defn}

We refer the reader to \cite[Example 2.1]{LT} for concrete examples of these combinatorial objects.

\begin{thm}[{\cite[Theorem 2.2]{LT}}]\label{T: signed Young rule} Let $(\alpha|\beta)\in\C^2(n)$. Then $M(\alpha|\beta)$ has a Specht filtration such that every Specht module $S^\lambda$ occurs as factors with multiplicity $s^\lambda_{\alpha|\beta}$.
\end{thm}

Using the signed Young Rule, we obtain the following lemma we shall need in the proof of Theorem~\ref{T: char of M}.

\begin{lem}\label{L: 1 sstd} Let $\alpha,\beta$ be partitions. Then we have $s^{\alpha\cont(1^{|\beta|})}_{\alpha|\beta}=1$.
\end{lem}
\begin{proof} Let $k=\ell(\alpha)$. The only semistandard $(\alpha\cont(1^{|\beta|}))$-tableau of type $(\alpha|\beta)$ is the colouring where the nodes in row $i$ are coloured by $\bb{c}_i$ if $1\leq i\leq k$ and, when $k+1\leq j\leq k+|\beta|$, the nodes $(j,1)$ are coloured by $\bb{d}_1,\bb{d}_2,\ldots$ in the obvious way.
\end{proof}

For a simple example, when $\alpha=(3,2)$ and $\beta=(1,1)$, the only semistandard $(3,2,1^2)$-tableau of type $((3,2)|(1,1))$ is shown below. \[\begin{ytableau}
  \bb{c}_1&\bb{c}_1&\bb{c}_1\\
  \bb{c}_2&\bb{c}_2\\
  \bb{d}_1\\
  \bb{d}_2
\end{ytableau}\] We are now ready to prove Theorem~\ref{T: char of M}.

\begin{proof}[Proof of Theorem~\ref{T: char of M}] The converse of the statement is easy to prove. We shall only prove the other implication. Assume that $M(\alpha|\beta)\cong M(\sigma|\tau)$ and suppose that $\alpha=\rho\cont(1^a)$ and $\sigma=\nu\cont(1^b)$ such that $\rho_{\ell(\rho)}\geq 2$ and $\nu_{\ell(\nu)}\geq 2$, i.e., both $\rho$ and $\nu$ have no parts of size 1. We have \[M(\rho|\gamma)\cong M(\alpha|\beta)\cong M(\sigma|\tau)\cong M(\nu|\delta),\] where $\gamma=\beta\cont(1^a)$ and $\delta=\tau\cont(1^b)$. Since signed Young permutation modules have trivial sources, they admit unique lift to trivial source modules (see \cite[2.6.3]{Benson}). As such the corresponding ordinary characters of $M(\rho|\gamma)$ and $M(\nu|\delta)$ are identical, i.e., $\ch(M(\rho|\gamma))=\ch(M(\nu|\delta))$. By signed Young Rule, we deduce that $\sum_{\lambda\vdash n} s^\lambda_{\rho|\gamma}\ch(S^\lambda)=\sum_{\lambda\vdash n} s^\lambda_{\nu|\delta}\ch(S^\lambda)$. Since ordinary characters are linearly independent, this shows that $s^\lambda_{\rho|\gamma}=s^\lambda_{\nu|\delta}$ for all $\lambda\in\P(n)$.

Let $\lambda=\rho\cont(1^{|\gamma|})$ and $\mu=\nu\cont(1^{|\delta|})$. By Lemma \ref{L: 1 sstd}, we have $1=s^\lambda_{\rho|\gamma}=s^\lambda_{\nu|\delta}$ and $1=s^\mu_{\rho|\gamma}=s^\mu_{\nu|\delta}$. Since there is a semistandard $\lambda$-tableau $\t$ of type $(\nu|\delta)$, we conclude that $\nu_1\leq \lambda_1=\rho_1$. Similarly, $\rho_1\leq \mu_1=\nu_1$, and hence $\rho_1=\nu_1$. So the $\rho_1$ nodes in the first row of $\t$ must be filled by the colour $\bb{c}_1$. Again, since $\t$ is semistandard, we conclude that $\nu_2\leq \rho_2$. Similarly, we have $\rho_2\leq\nu_2$, and hence $\rho_2=\nu_2$. Continue in this fashion, we conclude that $\rho_i=\nu_i$ for all $i\in\{1,\ldots,m\}$ where $m=\min\{\ell(\rho),\ell(\nu)\}$. If $\ell(\nu)>m=\ell(\rho)$, since $\t$ is semistandard again, we necessarily have $\nu_{m+1}\leq \lambda_{m+1}=1$. But $\nu_{m+1}\geq \nu_{\ell(\nu)}\geq 2$. The contradiction shows that $\ell(\nu)\leq \ell(\rho)$. Similarly, we have $\ell(\nu)\geq \ell(\rho)$ and hence $\nu=\rho$. Therefore, we have $\alpha=\rho\cont(1^a)$ and $\sigma=\rho\cont(1^b)$.

Tensor with the sign representation, we obtain $M(\beta|\alpha)\cong M(\tau|\sigma)$. Argue exactly the same way as the above, we also conclude that $\beta=\zeta\cont(1^c)$ and $\tau=\zeta\cont(1^d)$ for some non-negative integers $c,d$ and partition $\zeta$. The proof is now complete.
\end{proof}

Passing the result back to the Schur algebra case under the Schur functor, we obtain the following corollary. 


\begin{cor} Let $(\alpha|\beta),(\sigma|\tau)\in\P^2(n)$, and let $n\geq \dim_F E$. Then $S^\alpha E\otimes \bigwedge^\beta E\cong S^\sigma E\otimes \bigwedge^\tau E$ if and only if $\alpha=\rho\cont(1^a)$, $\sigma=\rho\cont (1^b)$, $\beta=\zeta\cont(1^c)$ and $\tau=\zeta\cont(1^d)$ for some non-negative integers $a,b,c,d$ and partitions $\rho,\zeta$.
\end{cor}
\begin{proof} Under the Schur functor, we have $M(\alpha|\beta)\cong f(K(\alpha|\beta))\cong f(K(\sigma|\tau))\cong M(\sigma|\tau)$. Now apply Theorem \ref{T: char of M}.
\end{proof}

\end{document}